\documentclass[reqno,17pt]{amsart}
\usepackage{amsfonts,amsmath,amsthm,amsxtra,amssymb,mathrsfs}
\usepackage{pb-diagram,pb-xy}
\usepackage[misc,geometry]{ifsym}
\usepackage[cmtip,arrow]{xy}
\usepackage[latin1]{inputenc}
\usepackage{graphicx,float}
\usepackage[colorlinks = true, citecolor = blue, urlcolor = blue]{hyperref}

\newtheorem{theorem}{Theorem}[section]

\newtheorem{proposition}[theorem]{Proposition}
\newtheorem{corollary}[theorem]{Corollary}
\theoremstyle{definition}
\newtheorem{definition}[theorem]{Definition}
\newtheorem{remark}[theorem]{Remark}

\newtheorem*{example}{Example}

\def\gor#1{\widetilde{#1}}
\def\tech#1{\widehat{#1}}
\def\bar#1{\overline{#1}}

\def\c{\,\circ}
\def\h(#1,#2){\mbox{Hom}(#1,#2)}
\def\t(#1,#2){\mbox{Tor}(#1,#2)}
\def\e(#1,#2){\mbox{Ext}(#1,#2)}

\def\N{\mathbb{N}}
\def\Z{\mathbb{Z}}

\def\R{\mathbb{R}}

\def\Q{\mathbb{Q}}
\def\C{\mathbb{C}}

\def\F{\mathbb{F}}
\def\T{\mathbb{T}}

\def\u{\mathbf{u}}
\def\v{\mathbf{v}}
\def\x{\mathbf{x}}
\def\y{\mathbf{y}}

\def\dgm{\textsf{dgm}}
\def\pers{\textsf{pers}}
\def\birth{\textsf{birth}}
\def\death{\textsf{death}}

\def\<{\langle}
\def\>{\rangle}

\newcommand{\Mat}[1]{{\begin{bmatrix}#1\end{bmatrix}}}

\title[Sliding Windows and Persistence]
{Sliding Windows and Persistence:  \\
An Application of Topological Methods to Signal Analysis}

\author[Jose Perea]{Jose A. Perea\;(\href{mailto:joperea@math.duke.edu}{\Letter})\;}
\address{Department of Mathematics,
         Duke University, Durham,
         North Carolina, USA.}
         \thanks{(\href{mailto:joperea@math.duke.edu}{\Letter}) Corresponding author. Email: {\tt joperea@math.duke.edu}. Phone: \texttt{+}\,1\;(919)\;660\,\texttt{-}\,2837.}
\email{joperea@math.duke.edu}

\author{John Harer}
\address{Departments of Mathematics, Computer Science
        and Electrical and Computer Engineering,
        Duke University, Durham, North Carolina, USA.}
        \thanks{Both authors were supported in part
        by DARPA under grants D12AP00001, D12AP00025-002,
        \indent and by the AFOSR under grant FA9550-10-1-0436.}
\email{john.harer@duke.edu}


\date{November 22$^{nd}$, 2013}


\begin{document}

\begin{abstract}

We develop in this paper a theoretical framework for the topological study of time series data.
Broadly speaking, we describe  geometrical  and topological properties
of  sliding window embeddings, as seen through the lens of persistent homology.
In particular, we show that maximum persistence at the point-cloud level can be used to quantify periodicity
at the signal level, prove structural and convergence theorems for the resulting persistence diagrams, and
derive estimates for their dependency on window size and embedding dimension.
We apply this methodology  to quantifying periodicity in synthetic
data sets, and compare the results with those obtained
using state-of-the-art methods in gene expression analysis.
We call this new method {\bf SW1PerS} which stands for
Sliding Windows and $1$-dimensional Persistence Scoring.

\end{abstract}

\maketitle

\section{Introduction}

Signal analysis is an enormous field.
There are many methods to study signals and many applications of
that study.
Given its importance, one might conclude that there is
little opportunity left for the development of totally new approaches to
signals.
Yet in this paper we provide a new way to find periodicity and quasi-periodicity
in signals.
The method is based on sliding windows (also known as time-delay reconstruction),
which have been used extensively
in both engineering applications and in dynamical systems.
But it adds a new element not applied before, which comes from
the new field of computational topology \cite{EH10}.

Persistent homology is a topological method for measuring the shapes
of spaces and the features of functions.
One of the most important applications of persistent homology is to
point clouds \cite{carlsson2009topology}, where shape is usually interpreted as the geometry of
some implicit underlying object near which the point cloud is sampled.
The simplest non-trivial example of this idea is a point cloud which has
the shape of a circle, and this shape is captured with $1$-dimensional
persistence.
The challenge in applying the method is that noise can reduce the
persistence, and not enough points can prevent the circular shape
from appearing.
It's also a challenge to deal with the fact that features come on all
scale-levels and can be nested or in more complicated relationships.
But this is what persistent homology is all about.

The idea of applying $1$D persistence to study time series arose
in our study of gene expression data \cite{DAOHHH12, DPH12}.
The first of these papers studied a variety of existing methods for
finding periodicity in gene expression patterns.
The motivation of that work was the search for gene regulatory networks
(more precisely, possible nodes of gene regulatory networks) that
control periodic processes in cells such as the cell division cycle,
circadian rhythms, metabolic cycles and periodic patterning in
biological development (lateral roots and somites).
The methods studied in \cite{DAOHHH12} were derived from a
number of fields including astronomy, geometry, biology and
statistics, and all were based on a direct study of the underlying
signal in either physical or frequency space.
The most successful methods are based  on finding
cosine-like behavior, a rather limited definition of periodicity.

In this paper and in \cite{DPH12} we look instead at the shape of
the sliding window point cloud, a totally different approach.
Of course the geometry of point clouds derived from other kinds
of data like images has been studied before
\cite{carlsson2008local, kantz2003},
but the current approach is quite different.
Our method understands periodicity as repetition of patterns, whatever these
may be, and quantifies this reoccurrence as the degree of
circularity/roundness in the generated point-cloud.
Thus, it is fundamentally agnostic.

\subsection{Previous Work}

The sliding window, or time-delay embedding, has been used mostly
in the study of dynamical systems to understand the nature of their attractors.
Takens' theorem  \cite{T81}  gives  conditions under which a smooth attractor
can be reconstructed from the observations of a function,
with bounds related to those of the Whitney Embedding Theorem.
This methodology has in turn been employed to test for non-linearity and
chaotic behavior in the dynamics of
ECG-EKG, EEG and MEG \cite{Stam05, Hundewale12}.
%

It has been recently demonstrated by de Silva et. al.
\cite{de-silva2012Topological}
 that combining time-delay embeddings with topological methods provides a
framework for parametrizing periodic systems.

Kantz and Schreiber  provide in \cite[Chapter~1]{kantz2003} a good source of examples of time delay embeddings used in real-world data sets.

\subsection{Our Contribution}
In the above applications, little of the topology and none of the geometry of the resulting
sliding window embedding has ever been used.
The novelty of our approach lies in our use of this geometry and topology through persistent homology.
We make this possible by showing that
maximum persistence, as a measure of ``roundness'' of the point-cloud,
occurs when the window size corresponds to the natural frequency of the signal.
 This means that $1$D persistence is an effective quantifier of periodicity and quasi-periodicity
 and can be used to infer properties of the signal.

\subsection{Outline}
In section \ref{motivation} we show a motivating example to illustrate our perspective.
In section \ref{background} we give a general introduction to persistent homology.
More on this topic can be found in \cite{EH10}.
In section \ref{approximation} we show
that sliding windows behave well under approximations, and give
explicit estimates at the point-cloud level.
Section \ref{trig} is devoted to studying the geometric structure of
sliding window embeddings from truncated Fourier series,
as well as their dependency on embedding dimension and window size.
In section \ref{secPersHomolTrunc} we prove results describing the
structure of persistent diagrams from sliding window embeddings.
We present in section \ref{secExamples} some examples of how our method
applies in the problem of quantifying  periodicity in time
series data.

\section{Definitions and Motivation}
\label{motivation}

Suppose that $f$ is a function defined on an interval of the real numbers.
Choose an integer $M$ and a real number $\tau$, both greater than $0$.
The sliding window embedding of $f$ based at $t \in \R$
 into $\R^{M+1}$ is the point
 \[
 SW_{M,\tau}f (t) = \Mat{f(t) \\ f(t + \tau) \\ \vdots \\ f(t + M\tau)}.
 \]
Choosing different values of $t$ gives a collection of points called a
{\bf sliding window point cloud for $f$}.
A critical parameter for this embedding is the {\bf window-size} $M\tau$.

\subsection{Motivation:}\label{secMotivation}
To motivate the approach we take in the paper,
let us begin with the following example.

\begin{example} Let $L\in \N$ and $f(t) = \cos(Lt)$.
Then
\begin{eqnarray*}
SW_{M,\tau} f(t) &=& \Mat{\cos(Lt) \\ \cos(Lt + L\tau) \\ \vdots \\ \cos(Lt + ML\tau)} \\ \\
&=& \cos(Lt) \Mat{1 \\ \cos(L\tau) \\ \vdots \\ \cos(LM\tau)} - \sin(Lt)\Mat{0 \\
\sin(L\tau)\\ \vdots \\ \sin(LM\tau)} \\ \\ &=& \cos(Lt)\u - \sin(Lt)\v
\end{eqnarray*}
and therefore $t \mapsto SW_{M,\tau} f(t)$ describes
a planar curve in $\R^{M+1}$, with winding number $L$, whenever $\u$
and $\v$  are linearly independent.
One can in fact
see how the shape of this curve changes as a function of $L$, $M$ and $\tau$.
Indeed, let \[A = \Mat{\|\u\|^2 & -\<\u,\v\> \\ \\ -\<\u,\v\>
& \|\v\|^2}\] which can be computed using  Lagrange's trigonometric formulae
\begin{eqnarray*}
\<\u,\v\> &=& \frac{1}{2}\sum_{m=0}^M \sin(2Lm\tau) = \frac{\sin(L(M+1)\tau) \sin(LM\tau)}{2\sin(L\tau)} \\ \\
\|\u\|^2 - \|\v\|^2 &=& \sum_{m=0}^M \cos(2Lm\tau) = \frac{\sin(L(M+1)\tau)\cos(LM\tau)}{\sin(L\tau)} \\ \\
\|\u\|^2 + \|\v\|^2 &=& M+1.
\end{eqnarray*}
It follows that $A$ is positive semi-definite (both its determinant
and trace ar non-negative).
This means the eigenvalues of $A$ are non-negative and real:
 $\lambda_1\geq\lambda_2\geq 0 $,
 and there is a  $2\times 2$ orthogonal matrix $B$  so that
 \[
 A = B^T\Lambda^2 B\quad \mbox{ where }\quad \Lambda = \Mat{\sqrt{\lambda_1} &
0 \\ 0 & \sqrt{\lambda_2}}.
\]
Therefore, if $\x(t) = \Mat{\cos(Lt) & \sin(Lt)}^\prime$  (here $\prime$ denotes transpose) then

\begin{eqnarray*}
\|SW_{M,\tau} f(t)\|^2 &=& \left \|\Mat{| &  | \\ \u & -\v \\ |   &| }\x(t)\right\|^2 \\ \\
&= &\big\<\x(t),A \;\x(t)\big\> \\ \\
&= & \big\<\Lambda B \; \x(t), \Lambda B \;\x(t)\big\>.
\end{eqnarray*}
Since $B$ is a rotation matrix, say by an angle $\alpha$,
then  the map
 \[SW_{M,\tau}f(t) \mapsto \Mat{\sqrt{\lambda_1} \cos(Lt +\alpha) \\
\\ \sqrt{\lambda_2} \sin(Lt + \alpha)} \] is an isometry.

In summary,  for $f(t) = \cos(Lt)$, the embedding
$t\mapsto SW_{M,\tau} f(t)$ describes an ellipse on the
plane $Span\{\u,\v\}$, whose shape (minor and major axes) is
determined by the square roots of the eigenvalues of $A$.
These eigenvalues can be computed explicitly as
\begin{eqnarray*}
\lambda_1 &=& \frac{(M+1) + \left|\frac{\sin(L(M+1)\tau)}{\sin(L\tau)}\right|}{2} \\ \\
\lambda_2 &=& \frac{(M+1) - \left|\frac{\sin(L(M+1)\tau)}{\sin(L\tau)}\right|}{2}.
\end{eqnarray*}  It follows that the ellipse is roundest when
$\lambda_2$ attains its maximum, which occurs if and only if
$L(M+1)\tau \equiv 0 \mod \pi$.
One such instance is \[M\tau = \left(\frac{M}{M+1}\right) \frac{2\pi}{L}\]
which  is when the window-size approximates the length of the period of $f(t)$.
In other words, the roundness of the sliding window point cloud for $f(t)=\cos(Lt)$
is maximized when the window-size is close to resonating with its natural frequency.

\end{example}
The previous example provides the following intuition: For a generic
function $f$, the degree to which the image of $SW_{M,\tau}f$ traces
a closed curve in $\R^{M+1}$ reflects how periodic $f$ is.
Moreover, if $f$ is periodic, then the roundness of $SW_{M,\tau}f$
defined as the largest radius of a ball in $\R^{M+1}$ so that the curve
\[
t\mapsto SW_{M,\tau}f(t)
\]
 is tangent to at least two points of
its equator, is maximized when the window-size $M\tau$ approaches
the period length.
The goal of this paper is to understand these relations.
\newpage

The geometry of the curve $t\mapsto SW_{M,\tau}f(t)$ can be quite
complicated, as shown in figure \ref{figSynthFunction}.
\begin{figure}[H]
\centering
\includegraphics[scale = .47]{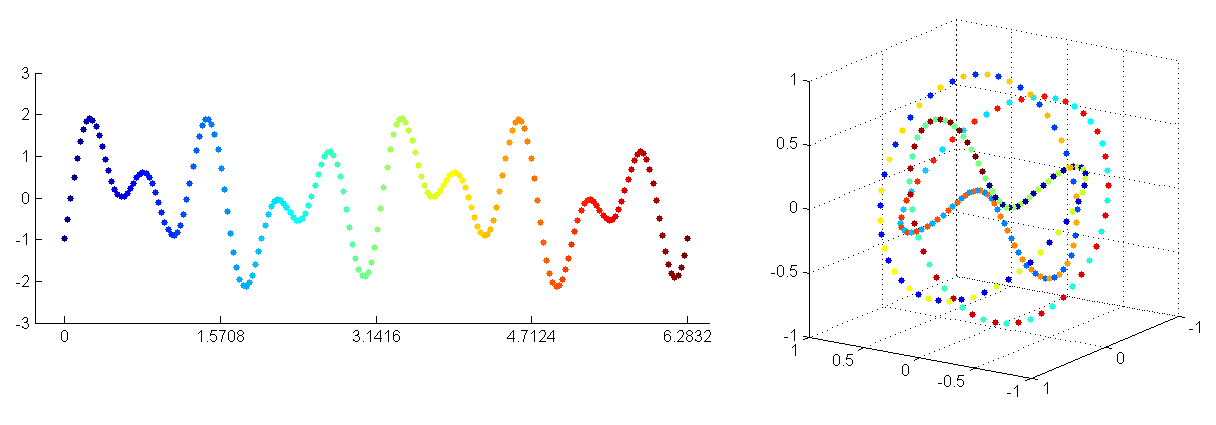}
\caption{From a periodic function to its sliding window point cloud. \textbf{Left:} A periodic function $f$. \textbf{Right:} Multidimensional scaling into $\R^3$ for $SW_{20,\tau} f$.  For each $t$, we use the same color for $f(t)$ and $SW_{20,\tau}f(t)$. Please refer to an electronic version for colors.}
\label{figSynthFunction}
\end{figure}
\noindent The 1-dimensional persistence
diagram for the Vietoris-Rips filtration on a finite sample
$\{SW_{M,\tau}f(t_1),\ldots, SW_{M,\tau} f(t_S)\}$, on the other hand,
is readily
computable \cite{javaplex, zomorodian2005computing} and its maximum persistence is a measure of
roundness as defined in the previous paragraph. We will review in section
\ref{background} the basic concepts behind persistent homology, and devote
the rest of the paper to understanding how the geometry of $SW_{M,\tau} f$
reflects properties of $f$ such as periodicity and period.

\subsection{Approach:}
With this motivation in mind, we now describe our approach: As we
have seen, understanding the algebraic properties of trigonometric
functions allows one to characterize the geometry of $SW_{M,\tau} f$
when $f$ is a trigonometric polynomial. This understanding, in turn,
can be bootstrapped using Fourier analysis and stability of persistence
diagrams, in an approximation step towards $SW_{M,\tau}$ of a
generic periodic function. In what follows, we will establish the
appropriate continuity results for approximation, as well as the necessary
structural results for persistence diagrams from sliding
window point clouds.

\section{Background: Persistent Homology}
\label{background}

In this section we define the key concepts that underlie the theory
of persistent homology for filtered simplicial complexes.
We give a terse introduction to simplicial homology, but more information
can be found in \cite{Hat02,Mun84}.

\subsection{Homology of Simplicial Complexes}

Let $K$ be a simplicial complex and $p$ a prime number.
Recall that this means that $K$ is a finite set of simplices that is closed
under the face relation and that two simplices of $K$
are either disjoint or intersect in a common face.
Let  $\F_p $ be the finite field with $p$ elements,
the $\F_p $ vector space generated by the
$k$-dimensional simplices of $K$ is denoted $C_k(K)$.
It consists of all \emph{$k$-chains}, which are finite formal sums
$c  =  \sum_j \gamma_j x_j$, with $\gamma_j \in  \F_p $
and each $x_j$ a $k$-simplex in $K$.
The boundary $\partial(x_j)$ is the alternating formal sum of the
$(k-1)$-dimensional faces of $x_j$ and the boundary of the chain $c$
is obtained by extending $\partial$ linearly
\begin{eqnarray*}
  \partial(c)  &=&  \sum_j  \gamma_j \partial(x_j).
\end{eqnarray*}

It is not difficult to check that $\partial \circ \partial = \partial^2 = 0$.
The $k$-chains that have boundary $0$ are called \emph{$k$-cycles};
they form a subspace $Z_k$ of $C_k$.
The $k$-chains that are the boundary of $(k+1)$-chains are
called \emph{$k$-boundaries} and form a subspace $B_k$ of $C_k$.
The fact that $\partial^2 = 0$ tells us that $B_k \subset Z_k$.
The quotient group $H_k (K) = Z_k / B_k$
is the \emph{$k$-th simplicial homology group}
of $K$ with $\F_p$-coefficients.
The rank of $H_k (K)$ is the \emph{$k$-th mod $p$ Betti number} of $K$ and
is denoted $\beta_k (K)$.
Since the prime $p$ will be clear from the context, we do not include it in
the notation.

When we have two simplicial complexes $K$ and $K'$, a
\emph{simplicial map} $f: K \to K'$ is a continuous map that
takes simplices to simplices and is linear on each.
A simplicial map induces a homomorphism on homology,
$f_*: H_k(K) \to H_k(K')$,
and homotopic maps induce the same homomorphism.
Homotopy equivalences of spaces induce isomorphisms on homology.
The simplicial approximation theorem tells us that
a continuous map of simplicial complexes can be approximated by a
simplicial map, so that it makes sense to talk about continuous
maps inducing homomorphisms on homology.

\subsection*{Persistence}
We next define persistence, persistent homology and the persistence diagram
 for a simplicial complex $K$.
A \emph{subcomplex} of $K$ is a subset of its simplices that is  closed
under the face relation.
A \emph{filtration} of $K$ is a nested sequence of subcomplexes
that starts with the empty complex and ends with the complete complex,
\[
  \emptyset = K_0 \subset K_1 \subset \ldots \subset K_m = K .
\]
A homology class $\alpha$ is \emph{born} at $K_i$ if it is not in the image
of the map induced by the inclusion $K_{i-1} \subset K_i$.
If $\alpha$ is born at $K_i$, we say that it \emph{dies entering} $K_j$
if the image of the map induced by $K_{i-1} \subset K_{j-1}$
does not contain the image of $\alpha$ but the image of the map
induced by $K_{i-1} \subset K_j$ does.
The \emph{persistence} of $\alpha$ is $j-i$.

We code birth and death information in the persistence diagrams,
one for each dimension.
The diagram $\dgm(k)$ has a point $(i,j)$ for every $k$-homology
class that is born at $K_i$ and dies entering $K_j$.
For most of the paper the homological dimension $k$ will  be clear from the context
or unimportant for the discussion.
To ease notation we will simply write $\dgm$ instead of $\dgm(k)$, and let
$\dgm_1$, $\dgm_2$ denote two $k$-persistence diagrams to be compared.
Sometimes we have a function $h$ that assigns a height or distance
to each sub complex $K_i$, and in that case we use the pair $(h(i),h(j))$.
Each diagram is now a multiset since classes can be born simultaneously
and die simultaneously.
We adjoin the diagonal $\Delta = \{(x,x):  x\geq 0 \}$ to each diagram,
and endow each point $(x,x)\in \Delta$ with countable multiplicity.

The Bottleneck distance between two persistence diagrams
$\dgm_1$ and $ \dgm_2$ is defined by
\[
d_B(\dgm_1,\dgm_2) =
\inf_{\phi} \sup_{x\in \dgm_1}  ||x - \phi(x)||_{\infty}
\]
where the infimum  is taken over all bijections $\phi:\dgm_1 \to \dgm_2$.
Note that such $\phi$ exist even if the number of points of $\dgm_1$ and $\dgm_2$ are
different since we have included the diagonal.

\subsection*{Rips Complex}
Let $X \subset \R^n$ be a compact set, for example a finite point cloud.
We define $d_X(y)$ to be the distance from the point $y \in \R^n$ to $X$.
We are interested in how the homology of the sub-level sets $X_r = d_X^{-1}([0,r])$
changes as we increase $r$.
To make this computationally feasible, we replace the continuous family
of spaces $X_r$ with a discrete family of approximations called the Rips complexes
defined as follows.
Fix $r \geq 0$, $R_r(X)$ is the simplicial complex whose vertices are the points of
$X$ and whose $k$-simplices are the  $k+1$ tuples $[x_0, \cdots, x_k]$ such that
the pairwise distances $||x_i - x_j||$ are less than or equal to $r$ for all $0 \leq i < j \leq k$.
Note that the edges determine the simplices of $R_r(X)$, a higher dimensional simplex
is added if and only if all its edges have been added,
and that the Rips construction makes sense for any metric space.

Since $R_r(X) \subset R_s(X)$ whenever $r < s$, the Rips complexes
form a filtration of $R_{\infty}$, which denotes the largest simplicial complex
having $X$ as its vertex set.
Changes occur at the finite set of $r$ values that are pairwise distances between points,
so we can work with just these $r_j$ to get the filtration
\[
X = R_0 \subset R_1 \subset \cdots \subset R_m,
\]
where $R_j(X) = R_{r_j}(X)$ and $R_{m} = R_{\infty}$.
We will use this filtered complex to study the persistence and the
persistence diagrams of the point cloud $X$.
We thus denote by $\dgm(X)$ the persistence diagram of the homology
filtration induced from the Rips filtration on $X$, where we use homology with
coefficents in $\F_p$.

A key property of persistence is that it is \emph{stable} \cite{CEH}.
In our context this means that if $X,Y$ are two point clouds and
$d_{H}$, $d_{GH}$ are the Hausdorff and Gromov-Hausdorff distances,
then

\begin{equation}\label{eqStability}
d_B(\dgm(X),\dgm(Y)) \leq 2d_{GH}(X,Y) \leq 2d_H(X,Y).
\end{equation}


\section{The Approximation Theorem}
\label{approximation}

In this section we show that one can study $SW_{M,\tau} f$  and the persistence of the point
cloud it generates for a generic function $f\in L^2(\T = \R / 2\pi \Z)$,
by using its Fourier Series approximation.
While it seems quite difficult to study $SW_{M,\tau} f$ directly,  it is not hard to understand
$SW_{M,\tau} \cos(nt)$ and  $SW_{M,\tau} \sin(nt)$,
so we will build our understanding of
the geometry of a general $SW_{M,\tau} f$ from these special cases using the
Fourier series of $f$.
To do this we will need to show that $SW_{M,\tau}$ behaves well under approximations
and that these  approximations work in the context of stability for persistence diagrams.

Let $C(X,Y)$ denote the set of continuous functions from $X$ to $Y$ equipped with the sup norm.
The sliding window embedding induces a mapping
\[
SW_{M,\tau}: C(\T,\R) \longrightarrow C(\T,\R^{M+1}).
\]
The first fact about this map that we need is the following:
\begin{proposition}\label{propSWisContinuous} Let $\T = \R / 2\pi \Z$.
Then for all $M\in \N$ and $\tau >0$, the mapping
 $SW_{M,\tau}: C(\T,\R) \longrightarrow C(\T,\R^{M+1})$
is a bounded linear operator with norm $\|SW_{M,\tau}\| \leq \sqrt{M+1}$.
\end{proposition}

\begin{proof}
Linearity of $SW_{M,\tau}$ follows directly from its definition.
To see that it is bounded, notice that for
every $f\in C(\T,\R)$ and  $t \in \T$ we have

\begin{eqnarray*}
\|SW_{M,\tau} f(t)\|^2_{\R^{M+1}} &=& |f(t)|^2 + |f(t+\tau)|^2 + \cdots + |f(t+M\tau)|^2 \\ \\
&\leq& (M+1)\|f\|^2_{\infty}
\end{eqnarray*}

\end{proof}

We now consider approximating a function $f$ by its Fourier polynomials and study how the
sliding windows behave in this context.
In particular, let

\[
f(t) = S_N f(t) + R_N f (t)\]
where
\[
 S_N f(t) =
 \sum_{n=0}^N a_n \cos(nt) + b_n \sin(nt) =
 \sum_{n = -N}^N \tech{f}(n) e^{int}
\]
 is the $N$-truncated Fourier series expansion of $f$, $R_N f$ is the remainder, and

\begin{equation}\label{eqRealFourCoeff}
 \tech{f}(n) = \left\{ \begin{array}{ll}
 \frac{1}{2} a_n - \frac{i}{2} b_n &\mbox{ if $n>0$,} \\ \\
 \frac{1}{2} a_{-n} + \frac{i}{2} b_{-n} &\mbox{ if $n<0$,} \\ \\
  a_0 &\mbox{ if $n=0$. }
       \end{array} \right.
\end{equation}
We can easily compute that
\[
SW_{M,\tau}f (t)= \sum\limits_{n=0}^N \cos(nt)\big(a_n \u_n + b_n \v_n\big) + \sin(nt)\big(b_n \u_n - a_n \v_n\big)+ SW_{M,\tau} R_N f(t)
\]
where
\[
\u_n = SW_{M,\tau} \cos(nt)\big|_{t = 0}\;\; \mbox{ and } \; \v_n = SW_{M,\tau}\sin(nt)\big|_{t = 0}.
\]
The vectors $\u_n$ and $\v_n$
form a fundamental basis out of which we can build our understanding
of the structure of the point clouds that sliding windows create.
We  introduce the notation:
\[
\phi_\tau(t) = \sum\limits_{n=0}^N \cos(nt)\big(a_n \u_n + b_n \v_n\big) + \sin(nt)\big(b_n \u_n - a_n \v_n\big),
\]
for the sliding window embedding for $S_N f(t)$.
Also, when $f,M$ and $N$ are clear from the context
we will simply write $\phi_\tau = SW_{M,\tau} S_N f$.

The next step is to find a bound on
the term $SW_{M,\tau} R_N f(t)$.
We will actually find a series of bounds, one for
each of the derivatives $f^{(k)} = \frac{d^k f}{dt^k}$, whenever they exist
and are continuous.

\begin{proposition}\label{propHausBound}
Let $k\in \N$. If   $f\in C^k(\T,\R)$ then for all $t\in \T$
\[\|SW_{M,\tau}f(t) - \phi_\tau(t)\|_{\R^{M+1}} \leq \sqrt{4k -2}\left\|R_N f^{(k)}\right\|_{2} \cdot \frac{\sqrt{M+1}}{(N+1)^{k - \frac{1}{2}}}\]
\end{proposition}
\begin{proof}

If $k\in \N$ and $f\in C^k(\T,\R)$, then integration by parts yields the well known identity
\[\left|\tech{f^{(k)}}(n)\right| = |n|^k \left|\tech{f}(n)\right|\]
for the length of $\tech{f^{(k)}}(n)$,
 the $n$-th complex Fourier coefficient of $f^{(k)}$, $n\in \Z$.
Thus for all $t\in \T$, the Cauchy-Schwartz inequality, Young's inequality
and Parseval's theorem together imply that
\begin{eqnarray*}
|R_N f(t)|  &\leq& \sum_{n = N+1}^\infty \frac{\left|\tech{f^{(k)}}(n)\right| + \left|\tech{f^{(k)}}(-n)\right|}{n^k} \\ \\
&\leq& \left(\sum_{n= N+1}^\infty \left(\left|\tech{f^(k)}(n)\right| + \left|\tech{f^{(k)}}(-n)\right|\right)^2\right)^{1/2}\cdot\left(\sum_{n= N+1}^\infty \frac{1}{n^{2k}}\right)^{1/2} \\ \\
&\leq&\left(2\sum_{|n|\geq N+1} \left|\tech{f^{(k)}}(n)\right|^2\right)^{1/2}\cdot\left(\int_{N+1}^\infty \frac{1}{x^{2k}} dx\right)^{1/2} \\ \\
&=& \sqrt{2}\left\|R_Nf^{(k)}\right\|_{2}\cdot \frac{\sqrt{2k -1}}{(N+1)^{k -\frac{1}{2}}}
\end{eqnarray*} and hence, by proposition \ref{propSWisContinuous}
\begin{eqnarray*}
\|SW_{M,\tau} f(t) - \phi_\tau(t)\|_{\R^{M+1}} &\leq& \sqrt{M+1}\|R_N f\|_\infty \\
&\leq& \sqrt{4k -2}\left\|R_N f^{(k)}\right\|_{2} \cdot \frac{\sqrt{M+1}}{(N+1)^{k - \frac{1}{2}}}
\end{eqnarray*}
\end{proof}

These bounds readily imply estimates for the Hausdorff distance between the sliding
window point clouds of $f$ and $S_N f$.
Indeed, let $X$ and $Y$ be the images of $T\subset \T$ through $SW_{M,\tau} f$
and $\phi_\tau$ respectively.
It follows that if $f\in C^k(\T,\R)$ and
\[\epsilon > \sqrt{4k -2}\left\|R_N f^{(k)}\right\|_{2} \frac{\sqrt{M+1}}{(N+1)^{k - \frac{1}{2}}}\]
then
$X \subset Y^\epsilon$, $Y\subset X^\epsilon$ and therefore $d_H(X,Y) \leq \epsilon $.
Letting $\epsilon$ approach its lower bound and using the stability of  $d_B$ with respect to $d_H$ (equation \ref{eqStability}), we obtain the relation
\[d_B\big(\dgm(X),\dgm(Y)\big)\; \leq\; 2\sqrt{4k -2}\left\|R_N f^{(k)}\right\|_{2} \frac{\sqrt{M+1}}{(N+1)^{k - \frac{1}{2}}}\]

As described in the introduction, the maximum persistence of $\dgm(X)$ will serve
to quantify the periodicity of $f$ when measured with sliding windows of length $M\tau$. By the maximum persistence of a diagram $\dgm$ we mean the following

\begin{definition} Let $(x,y)\in \dgm$  and define $\pers(x,y) = y-x$ for $(x,y) \in \R^2$, and as $\infty$ otherwise.
We let \[mp(\dgm) = \max_{\x\in\dgm} \pers(\x) \] denote the maximum persistence of $\dgm$.
\end{definition}

\begin{remark}\label{rmkMaxPersDist2Diagonal} If $\dgm_\Delta$ denotes the diagram with the
diagonal  as underlying set,   each point endowed with countable multiplicity, then
 \[mp(\dgm) = 2 d_B(\dgm,\dgm_\Delta).\]
Indeed, for any bijection
$\phi: \dgm \longrightarrow \dgm_\Delta$ and every $\x \in \dgm$
\[\|\x - \phi(\x)\|_\infty \geq \frac{1}{2} \pers(\x)\] with equality if and only if $\phi(x,y) = \left(\frac{x+y}{2},\frac{x+y}{2}\right)$. Thus
\[\max_{\x \in \dgm}\|\x - \phi(\x)\| \geq \frac{1}{2} mp(\dgm)\] and therefore
$
d_B(\dgm , \dgm_\Delta) = \min\limits_{\phi} \max\limits_{\x \in \dgm} \|\x - \phi(\x)\|
\geq \frac{1}{2}mp(\dgm).
$ For the reverse inequality, notice that the map \[(x,y) \mapsto \left(\frac{x+y}{2}, \frac{x+y}{2}\right)\] extends to a bijection $\phi_0 : \dgm \longrightarrow \dgm_\Delta$ of multisets, such that for all
$\x\in \dgm$ one has $ \|\x - \phi_0(\x)\|_\infty = \frac{1}{2} \pers(\x)$.
\end{remark}

We  summarize the results of this section in the following theorem:

\begin{theorem}[\textbf{Approximation}]\label{thmApproximation} Let $T\subset \T$, $f\in C^k(\T,\R)$, $X = SW_{M,\tau} f(T)$ and $Y = SW_{M,\tau}S_N f(T)$. Then

\begin{enumerate}\item \[ d_H(X,Y)
\leq \sqrt{4k -2}\left\|R_N f^{(k)}\right\|_{2} \frac{\sqrt{M+1}}{(N+1)^{k - \frac{1}{2}}}\]
\item \[ \left|mp\big(\dgm(X)\big) - mp\big(\dgm(Y)\big)\right| \leq 2d_B\big(\dgm(X),\dgm(Y)\big)\]

\item \[ d_B\Big(\dgm(X),\dgm(Y)\Big)
\leq 2\sqrt{4k -2}\left\|R_N f^{(k)}\right\|_{2} \frac{\sqrt{M+1}}{(N+1)^{k - \frac{1}{2}}}\]
\end{enumerate}
\end{theorem}

It follows that the persistent homology
of the sliding window point cloud of a function $f\in C^k(\T,\R)$ can,
in the limit, be understood in terms of
that of its truncated Fourier series.

\begin{remark} Regarding the hypothesis of $f$ being at least $C^1$,
Proposition \ref{propHausBound} (which is the basis of the Approximation Theorem, \ref{thmApproximation}) only uses that $f^{\prime}  \in L^2(\T)$,
thus everything up to this point (and in fact, for the rest of the paper)
holds true for functions in the Sobolev space $W^{1,2}(\T)$.
The reason why we have
phrased the results in terms of the spaces $C^k(\T)$ is because it provides
the following interpretation: If the function $f$ has certain degree of
niceness, then one should expect the approximation of the persistence diagrams
from $SW_{M,\tau} f$ by those of $SW_{M,\tau} S_N f$ to improve at an explicit rate.
Moreover, the nicer the function the better the  rate.

Another function space for which  our arguments apply is the set of H\"{o}lder continuous functions with  exponent $\alpha \in\left(\frac{1}{2},1\right)$. Indeed, if for
such an $f$  one considers the Fej\'er approximation
\[\sigma_N f (t) = \sum_{|n|\leq N} \left(1 - \frac{|n|}{N+1}\right) \tech{f}(n)e^{int}\] then (see \cite[Theorem 1.5.3]{pinsky})
\[\|\sigma_N f - f\|_\infty \leq \frac{ C_\alpha K_\alpha }{N^\alpha}\]
where $K_\alpha$ is the H\"older constant of $f$ and $C_\alpha$ is a constant
depending solely on $\alpha$.
Hence one gets the following version of Proposition \ref{propHausBound}:
For every $t\in \T$
\[\big\|SW_{M,\tau} f(t) - SW_{M,\tau} \sigma_N f(t) \big\|_{\R^{M+1}} \leq C_\alpha K_\alpha \frac{\sqrt{M+1}}{N^\alpha}\]
and the corresponding version of the Approximation Theorem follows.
Later on (e.g. in Theorem \ref{thmWindSize2MaxPers}) we will use some
bounds in terms of $\|f^\prime\|_2$ and $\|S_N f^\prime\|_2$.
Adapting results involving said bounds to the Holder-continuous setting
requires some work: One can
use the fact that every Holder function has a Lipschitz approximation,
and then invoke  Rademacher's theorem.
We leave the details to the interested reader.

\end{remark}

\section{The Geometric Structure of $SW_{M,\tau} S_N f$}
\label{trig}

We now turn our attention to  the sliding window construction when applied to the truncated
Fourier series of a periodic function. More specifically, we study
the geometric structure of  the sliding window point cloud,
and its  dependency  on $\tau$, $N$ and $M$.

Our focus on geometry contrasts with methods used by others to determine
$\tau$ and $M$. Traditionally $M\tau$, the window size, is estimated using the autocorrelation function \cite{kim1999nonlinear}.
While $M$ is sometimes estimated directly using the method of
false nearest neighbors \cite{kantz2003}.

\subsection{Dimension of the  Embedding }

One way of interpreting the dimension of the embedding, $M+1$, is as the level of detail
(from the function)  one hopes to capture with the sliding window representation.
Given the advantages of  a description
which is as detailed as possible, it can be argued that large dimensions are desirable. From
a computational perspective, however, this is a delicate point as  our ultimate goal
is to compute the persistent homology of the
associated sliding window point cloud. Indeed, as  the dimension of the embedding grows,
it follows that the point cloud needs to be (potentially) more densely populated.
This causes the size of the
Rips complex to outweigh the computational resources, making the persistent homology calculation unfeasible.

While there has been considerable progress on dealing with the size
problem of the Rips complex \cite{nanda2013Morse}, it is important to have a sense of the
amount of retained information  given the computational constraints on the  embedding dimension. Fortunately, when dealing with
trigonometric polynomials the answer is clear: One loses   no information
if and only if the embedding dimension is greater than twice
the maximum frequency. Indeed, recall the linear decomposition
\[
SW_{M,\tau}S_Nf (t)= \sum\limits_{n=0}^N \cos(nt)\big(a_n \u_n + b_n \v_n\big)
+ \sin(nt)\big(b_n \u_n - a_n \v_n\big)
\]
where \[\u_n = SW_{M,\tau} \cos(nt)\big|_{t = 0}\;\; \mbox{ , } \; \v_n
= SW_{M,\tau}\sin(nt)\big|_{t = 0}\] and $a_n,b_n$ are as defined in equation \ref{eqRealFourCoeff}.
Since the angles between the $\u_n$'s and the
$\v_m$'s, as well as their norms can be determined
from $M$ and $\tau$ (see Example in section \ref{secMotivation}),
then $S_N f$ can be recovered from $SW_{M,\tau} S_N f$ if
the $\u_n$'s and the $\v_m$'s are linearly independent. This is the sense in which we say that there is no loss of information.

\begin{proposition}\label{propLinInd} Let $M\tau < 2\pi$. Then $\u_0,\u_1,\v_1,\ldots, \u_N,\v_N$
are linearly independent if and only if $M \geq 2N$.
\end{proposition}

\begin{proof} If $2N + 1$ vectors in $\R^{M+1}$ are linearly independent, it readily follows
that \mbox{$2N \leq M$}. Let us assume now that $\u_0,\u_1,\v_1,\ldots,\u_N,\v_N$ are linearly
dependent and let us show that $2N > M$, or equivalently that $2N \geq M+1$. Indeed, let
$\gamma_0,\beta_0,\ldots,\gamma_N,\beta_N \in \R$ be scalars not all zero (set $\beta_0 = 0$)
so that \[\gamma_0\u_0 +  \beta_0 \v_0 + \cdots + \gamma_N \u_N + \beta_N \v_N = \mathbf{0}. \]
That is,  for all $m = 0,\ldots, M$ we have

\[
0 = \sum_{n = 0}^N \gamma_n \cos(nm\tau) + \beta_n \sin(nm\tau)
= Re\Big(\sum_{n=0}^N (\gamma_n - i\beta_n)e^{inm\tau} \Big).\]

\noindent Let $\xi_m = e^{im\tau}$,\;\; $p(z) = \sum\limits_{n=0}^N (\gamma_n + i\beta_n)z^n$,\;\; $\bar{p}(z) =
\sum\limits_{n=0}^N (\gamma_n - i\beta_n)z^n$\;\; and
\[q(z) = z^N\cdot\left(\bar{p}(z) + p\left(\frac{1}{z}\right)\right)_{\mbox{\large.}}\]
It follows that $q(z)$ is a non-constant complex polynomial of degree
at most $2N$, and that for $m=0,\ldots,M$ we have $0 = Re(\bar{p}(\xi_m))$.
This implies that
\begin{eqnarray*}
q(\xi_m) &=& (\xi_m)^N \left(\bar{p}(\xi_m) + p\left(\frac{1}{\xi_m}\right)\right) \\
&=& (\xi_m)^N \left(\bar{p}(\xi_m) + p\Big(\bar{\xi_m}\Big)\right) \\
&=& 2(\xi_m)^N\; Re\Big(\bar{p}(\xi_m)\Big)\\ &=& 0
\end{eqnarray*}
and therefore $\xi_0,\ldots, \xi_M$ are roots of $q(z)$.
Since
$M\tau < 2\pi$ then  $\xi_0,\xi_1,\ldots, \xi_M$
 are distinct, and we have that
\[M+1 \leq degree\big(q(z)\big) \leq 2N.\]
\end{proof}

It is useful to contrast Proposition \ref{propLinInd} with two important
results in signal analysis: Takens' theorem from dynamical systems \cite{T81},
and the Nyquist-Shannon sampling theorem from information
theory \cite{Shannon}. Takens' theorem gives sufficient conditions on the
length of a sequence of observation, so that the resulting embedding  recovers the
topology of a smooth attractor in a chaotic dynamical system. The
aforementioned condition is that the dimension of the embedding should be
greater than twice (an appropriate notion of) that of the attractor. The
Nyquist-Shannon sampling theorem, on the other hand, contends that a
band-limited signal can be recovered (exactly) from a sequence of observations
whenever the sampling frequency is greater than twice the position,
in the frequency domain, of the limiting
band. The conclusion: In the case of
trigonometric polynomials and the sliding window construction, the usual
sufficient condition on dimension of the embedding and maximum frequency is
also necessary.

\vspace{.4cm}
\noindent \textbf{Important Assumption:} Unless otherwise stated, given $N\in \N$ we will always set $M = 2N$,
and require $\tau > 0$ to be so that $M\tau < 2\pi$.
\vspace{.1cm}


%
%

\subsection{Window Size and Underlying Frequency}

As we saw in the Motivation Section \ref{motivation}, the sliding window
point cloud for $\cos(Lt)$ describes a planar ellipse which is roundest
when $\|\u\| - \|\v\|= \<\u,\v\>=0$, or equivalently, when \[L(M+1)\tau \equiv
0\mbox{ (mod $\pi$)}.\] This  uncovers a fundamental relation between window size,
1-dimensional persistence and underlying frequency: The maximum persistence
of the sliding window point cloud from $\cos(Lt)$ is largest when the
window size $M\tau$ is proportional to the underlying frequency $\frac{2\pi}{L}$,
with proportionality constant $\frac{M}{M+1}$.

For the case of the truncated Fourier series $S_N f$ from a  periodic function $f$,
we will see shortly that if the same proportionality relation between  window size
and underlying frequency holds then
\begin{equation}\label{eqLinDecomp}
SW_{M,\tau}S_Nf (t)= \sum\limits_{n=0}^N \cos(nt)\big(a_n \u_n + b_n \v_n\big)
+ \sin(nt)\big(b_n \u_n - a_n \v_n\big)
\end{equation}
is a linear decomposition into mutually orthogonal vectors.
 We begin with the now
familiar case of the restriction to $Span\{\u_n , \v_n\}$.

\begin{proposition}\label{propOrthIff} For $n\geq 1$, $\<\u_n, \v_n\> = \|\u_n\|^2
- \|\v_n\|^2 =0$ if and only if \[n(M+1)\tau \equiv 0  \;\;(\mbox{mod } \pi).\]
\end{proposition}
\begin{proof}
\begin{eqnarray*}
\<\u_n,\v_n\> &=& \sum_{m=1}^M \cos(nm\tau)\sin(nm\tau) = \frac{1}{2}\sum_{m=1}^M \sin(2nm\tau) \\
&=& \frac{1}{2}Im\left(\sum_{m=1}^M z_{2n\tau}^m\right), \;\;\mbox{where } z_\theta = e^{i\theta} \\
&=& \frac{1}{2}Im\left(\frac{1 - z_{2n(M+1)\tau}}{1 - z_{2n\tau}} -1 \right)\\
&=& \frac{1}{2}Im\left(\frac{1 - z_{2n(M+1)\tau}}{1 - z_{2n\tau}}\right)\\
\|\u_n\|^2 - \|\v_n\|^2 &=& \sum_{m=0}^M \cos^2(nm\tau) - \sin^2(nm\tau)\\
&=& Re\left(\frac{1 - z_{2n(M+1)\tau}}{1-z_{2n\tau}}\right)
\end{eqnarray*} and therefore \[4\<\u_n,\v_n\>^2 + (\|\u_n\|^2 - \|\v_n\|^2)^2 =
\left\|\frac{1 - z_{2n(M+1)\tau}}{1 - z_{2n\tau}}\right\|^2.\] It follows that
$\<\u_n,\v_n\> = \|\u_n\|^2 - \|\v_n\|^2 = 0 $ if and only if $z_{2n(M+1)\tau} =1$,
which holds true if and only if $n(M+1)\tau \equiv 0$ (mod $\pi$).
\end{proof}

It can be checked that $nM\tau \equiv 0$ (mod $\pi$) also yields $\<\u_n,\v_n\> = 0$,
but letting $n(M+1)\tau \equiv 0$ (mod $\pi$)  implies that $a_n \u_n + b_n \v_n$
is perpendicular to $b_n \u_n - a_n \v_n$ for all $a_n,b_n \in \R$. Now, in order to
extend the perpendicularity results to components from different harmonics, we will use the following:

\begin{definition}\label{defLperiodic} We say that a  function $f$ is   $L$-periodic  on
$[0,2\pi]$, $L\in \N$, if  \[f\left(t + \frac{2\pi}{L}\right) =f(t)\] for all $t$.
\end{definition}

\begin{remark}\label{rmkLperiodic} If $f$ is  an $L$-periodic function, $a_n$ and $b_n$
are its $n$-th real Fourier coefficients (see equation \ref{eqRealFourCoeff}), and we let $a_n + ib_n = r_n e^{i\alpha_n}$,
with $\alpha_n = 0$ whenever $r_n = 0$;  then  $r_n \neq 0$ implies $n \equiv 0$ (mod $L$).
Indeed,  $g(t) = f(t/L)$ is a 1-periodic function and therefore  has a Fourier series expansion
\[g(t) = \sum_{r=0}^\infty a^\prime_r \cos(rt) + b^\prime_r\sin(rt)\] with equality  almost everywhere.
Thus \[f(t) = g(tL) = \sum_{r=0}^\infty a^\prime_r \cos(rLt) + b^\prime_r\sin(rLt)=
\sum_{n=0}^\infty a_n \cos(nt) + b_n\sin(nt)\] for almost every $t$,
and the result follows from the uniqueness of the Fourier expansion in $L^2(\T)$.
\end{remark}

We are now ready to see that the potentially non-zero terms in the linear decomposition of
$SW_{M,\tau} S_N f$ (equation \ref{eqLinDecomp}),  can be made mutually orthogonal by choosing the window size proportional
to the underlying frequency, with proportionality constant $\frac{M}{M+1}$.

\begin{proposition}\label{propLperd2WindowSize} Let $f$ be $L$-periodic, and let
$\tau = \frac{2\pi}{L(M+1)}$. Then the vectors in  \[\{\u_n,\v_n \; | \; 0\leq n \leq N,\;
n \equiv 0 \; \mbox{ (mod $L$)}\}\]
are mutually orthogonal, and we have $\|\u_n\| = \|\v_n\| = \sqrt{\frac{M+1}{2}}$
for $n\equiv 0$ (mod $L$).
\end{proposition}
\begin{proof} Let $k = pL$ and $n = qL$.
If $k=n$, it follows from Proposition \ref{propOrthIff} that $\<\u_n,\v_n\> = 0 $
and
\begin{eqnarray*}
\|\u_n\|^2 = \|\v_n\|^2 &=& \frac{\|\u_n\|^2 + \|\v_n\|^2}{ 2} =\frac{1}{2}\sum_{m=0}^M \cos(nm\tau)^2 + \sin(nm\tau)^2 \\
&=& \frac{M+1}{2}.
\end{eqnarray*}
Let us assume now that $p\neq q$. If we let $z_\theta = e^{i \theta}$,
$\theta \in \R$, then
\begin{eqnarray*}
\<\u_n,\u_k\> &=& \sum_{m=0}^M \cos(nm\tau)\cos(km\tau) \\
&=& \frac{1}{2}\sum_{m=0}^M \cos((n-k)m\tau) + \cos((n+k)m\tau) \\
&=& \frac{1}{2}Re \left(\frac{1 - z_{(n-k)(M+1)\tau}}{1 - z_{(n-k)\tau}} + \frac{1 - z_{(n+k)(M+1)\tau}}{1 - z_{(n+k)\tau}}\right)\\
&=& \frac{1}{2}Re \left(\frac{1 - z_{(q - p)2\pi}}{1 - z_{(n-k)\tau}} + \frac{1 - z_{(q+p)2\pi}}{1 - z_{(n+k)\tau}}\right) = 0.
\end{eqnarray*}

\noindent Notice that \[0< \min\{|n-k|,|n+k|\} \leq \max\{|n-k|,|n+k|\} \leq 2N \leq M < \frac{2\pi}{\tau}\]
implies that the denominators are never zero. Similarly

\begin{eqnarray*}
\<\u_n,\v_k\> &=& \sum_{m=1}^M \cos(nm\tau) \sin(km\tau) \\
&=& \frac{1}{2}\sum_{m=1}^M \sin((n+k)m\tau) - \sin((n-k)m\tau) \\
&=& \frac{1}{2}Im\left(\frac{1 - z_{(q + p)2\pi}}{1 - z_{(n+k)\tau}} - \frac{1 - z_{(q-p)2\pi}}{1 - z_{(n-k)\tau}}\right) = 0 \\ \\
\<\v_n,\u_k\> &=& \frac{1}{2}Im\left(\frac{1 - z_{(p + q)2\pi}}{1 - z_{(k+n)\tau}} - \frac{1 - z_{(p-q)2\pi}}{1 - z_{(k-n)\tau}}\right) = 0 \\ \\
\<\v_n, \v_k\> &=& \frac{1}{2}Re\left(\frac{1 - z_{(q - p)2\pi}}{1 - z_{(n-k)\tau}} - \frac{1 - z_{(q+p)2\pi}}{1 - z_{(n+k)\tau}}\right)=0.
\end{eqnarray*}
\end{proof}

When computing persistent homology it is sometimes advantageous to
pointwise  center and normalize the set of interest.
The next theorem
describes the result of such operations on the sliding window point
cloud for $SW_{M,\tau} S_N f$, when $f$ is $L$-periodic and $L(M+1)\tau = 2\pi$.

\begin{theorem}[\textbf{Structure}]\label{thmStructure} Let $C: \R^{M+1} \longrightarrow \R^{M+1}$ be the centering map
\[C(\x) = \x - \frac{\<\x,\mathbf{1}\>}{\|\mathbf{1}\|^2}\mathbf{1}\;\; \mbox{ where }\;\;  \mathbf{1} = \Mat{1 \\ \vdots \\ 1} \in \R^{M+1}.\]
If $f$ is $L$-periodic, $L(M+1)\tau = 2\pi$ and $\phi_\tau = SW_{M,\tau} S_N f$, then

\begin{enumerate}
\item \[\phi_\tau(t)= \tech{f}(0)\cdot\mathbf{1} + C(\phi_\tau(t)) \]
\item \[\left\|C\big(\phi_\tau(t)\big)\right\| = \sqrt{M+1}\left(\left\|S_N f\right\|_2^2 - \tech{f}(0)^2\right)^{1/2} \]
\item There exists an orthonormal set \[\left\{\gor{\x}_n , \gor{\y}_n \in \R^{M+1}\; \Big|\; 1\leq n \leq N, \; n\equiv 0 \;(\mbox{mod }L)\right\}\] such that
\begin{equation}\label{eqOrthDecomp}
    \varphi_\tau(t) = \frac{C\left(\phi_\tau(t)\right)}{\|C\left(\phi_\tau(t)\right)\|}= \sum_{\scriptsize\shortstack{$n=1$\\ $n\equiv 0$ (mod $L$)}}^N \gor{r}_n\big(\cos(nt)\gor{\x}_n + \sin(nt)\gor{\y}_n\big)
\end{equation}
    where \[\gor{r}_n = \frac{2\left|\tech{f}(n)\right| }{\sqrt{\|S_N f\|^2_2 - \tech{f}(0)^2}}\]

\end{enumerate}
\end{theorem}

\begin{proof}
If $f$ is an $L$-periodic function on $[0,2\pi]$ and $L(M+1)\tau = 2\pi$,
then Remark \ref{rmkLperiodic} and Proposition \ref{propLperd2WindowSize}
imply that for all $t\in \R$

\begin{eqnarray*}
\phi_\tau (t) &=& \sum_{\scriptsize\shortstack{$n=0$\\ $n \equiv 0$ (mod $L$)}}^N \cos(nt)\big(a_n \u_n + b_n \v_n\big) + \sin(nt)\big(b_n \u_n - a_n \v_n\big) \\
&=& \sum_{\scriptsize\shortstack{$n=0$ \\ $n \equiv 0$ (mod $L$)}}^N r_n\big(\cos(nt) \x_n + \sin(nt) \y_n\big)
\end{eqnarray*} is a linear combination  of the mutually orthogonal vectors  \mbox{$\x_n = \cos(\alpha_n)\u_n + \sin(\alpha_n)\v_n$} and $\y_n = \sin(\alpha_n)\u_n - \cos(\alpha_n)\v_n$.

Moreover, from Proposition \ref{propLperd2WindowSize} we have that if $n\geq 1$ is so that $n \equiv 0$ (mod $L$) then
$\|\x_n\| = \|\y_n\| = \sqrt{\frac{M+1}{2}}$. It follows that if \[\gor{\x}_n = \frac{\x_n}{\|\x_n\|}, \;\; \gor{\y}_n = \frac{\y_n}{\|\y_n\|}\] then
\[\phi_\tau (t) = \big(a_0\sqrt{M+1}\big) \frac{\mathbf{1}}{\|\mathbf{1}\|}\;\;\;\; + \sum_{\scriptsize\shortstack{$n=1$\\ $n\equiv 0$ (mod $L$)}}^N\sqrt{\frac{M+1}{2}}r_n\big(\cos(nt)\gor{\x}_n + \sin(nt) \gor{\y}_n\big)\] is a linear decomposition of $\phi_\tau(t)$ in terms of the orthonormal set \[\left\{\frac{\mathbf{1}}{\|\mathbf{1}\|}, \; \gor{\x}_n , \gor{\y}_n\; \Big|\; 1\leq n \leq N, \; n\equiv 0 \;(\mbox{mod }L)\right\}.\]

\noindent Hence $C\big(\phi_\tau(t)\big) \;\;\;= \sum\limits_{\scriptsize\shortstack{$n=1$\\ $n\equiv 0$ (mod $L$)}}^N\sqrt{\frac{M+1}{2}}r_n\big(\cos(nt)\gor{\x}_n + \sin(nt) \gor{\y}_n\big)$ and therefore

\begin{eqnarray*}
\varphi_\tau(t) &=& \frac{C\big(\phi_\tau(t)\big)}{\big\|C\big(\phi_\tau(t)\big)\big\|} \\
&=& \sum_{\scriptsize\shortstack{$n=1$\\ $n\equiv 0$ (mod $L$)}}^N \frac{r_n}{\sqrt{r_1^2 + \cdots + r_N^2}}\big(\cos(nt)\gor{\x}_n + \sin(nt)\gor{\y}_n\big)
\end{eqnarray*}
which we write as
\[
\varphi_\tau(t) = \sum_{\scriptsize\shortstack{$n=1$\\ $n\equiv 0$ (mod $L$)}}^N \gor{r}_n\big(\cos(nt)\gor{\x}_n + \sin(nt)\gor{\y}_n\big)\;\;\;\;,\;\;\;\; \sum\limits_{n=1}^N \gor{r}_n^2 = 1.
\]
The result follows  from the identities
$r_n = 2\left|\tech{f}(n)\right| = \left|\tech{f}(n)\right|+ \left|\tech{f}(-n)\right|$, $n\geq 1$.
\end{proof}

Theorem \ref{thmStructure} allows us to paint a very clear
geometric picture of the centered and normalized sliding window
point cloud for $S_N f$ (see equation \ref{eqOrthDecomp}). Indeed, if
$S^1(r)\subset \C$ denotes the circle of radius $r$ centered at zero,
then $t\mapsto \varphi_\tau (t)$ can be regarded as the curve in
the $N$-torus \[\mathcal{T} = S^1(\gor{r}_1)\times \cdots \times S^1(\gor{r}_N)\] which
when projected onto  $S^1(\gor{r}_n)$, $\gor{r}_n >0$, goes around
$n$ times at a constant speed. Another interpretation, in terms of flat
(polar) coordinates, is as the image through the quotient map
\[ \R^N = \R \times \cdots \times \R \;\longrightarrow \;\left(\R /\gor{r}_1\Z\right)
\times \cdots\times \left(\R /\gor{r}_N\Z\right) \]
of the line segment in $\R^N$ joining
$(0,0,\ldots, 0)$ and $(\gor{r}_1,2\gor{r}_2,\ldots,N\gor{r}_N)$.
Figure \ref{figFlatCoordinates} depicts $\varphi_\tau(t)$ inside  $\mathcal{T}$ for
$N=3$.

\begin{figure}[htb]
\centering
\includegraphics[scale = .47]{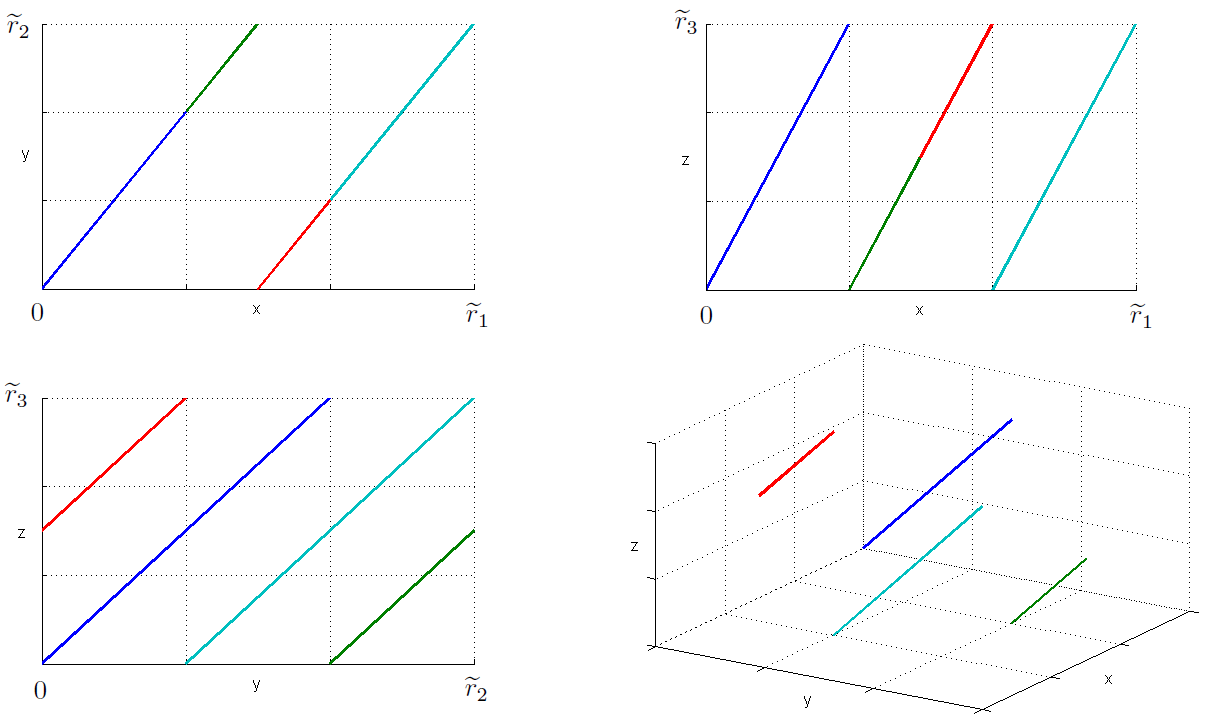}

\caption{The curve $\varphi_\tau(t)$, in colors, with respect to its flat coordinates $(t,2t,3t)\in(\R/\gor{r}_1\Z)\times (\R/\gor{r}_2\Z)\times (\R/\gor{r}_3\Z)$. Please refer to an electronic version for colors. \textbf{Bottom Right:} $\varphi_\tau(t)$ in the fundamental domain $[0,\gor{r}_1)\times [0,\gor{r}_2)\times  [0,\gor{r}_3)$. \textbf{Top Left:} Projection onto the $xy$-plane. \textbf{Top Right:} Projection onto the $xz$-plane. \textbf{Bottom left:} Projection onto the $yz$-plane.}\label{figFlatCoordinates}
\end{figure}

\section{The Persistent Homology of $\varphi_\tau$ and $SW_{M,\tau}f$}\label{secPersHomolTrunc}
The structural observations from the previous section, as well as the
approximation results from Section \ref{approximation}, set the stage for
understanding the persistent homology of the image of $\phi_\tau$
(or rather of
$\varphi_\tau$) and how it relates to that of $SW_{M,\tau} f$.

\subsection{Some convergence results} Let
$T\subset \T$,  and let $SW_{M,\tau}f(T)$ and
$\phi_\tau(T)$ be the images of $T$ through $SW_{M,\tau}f$ and $\phi_\tau$
respectively. An immediate consequence of Proposition \ref{propHausBound}
is that as $N$ (and thus $M =2N$) gets larger, $\phi_\tau(T)$ gets
closer to $SW_{M,\tau} f(T)$ with respect to the Hausdorff metric on
subspaces of $\R^\infty$. Here $\R^\infty$ denotes the set of sequences
$x= (x_k)_{k\in \N}$, $x_k \in \R$, so that $x_n = 0$ for all $n\geq N_0$,
and some $N_0 = N_0(x) \in \N$. We endow $\R^\infty$ with the $L^2$ metric, and regard
$SW_{M,\tau}f(t)$, $t\in T$,  as an element of $\R^\infty$ by identifying
it with \[\big(f(t), f(t+\tau),\ldots, f(t+M\tau),0,0,\ldots\big)\in \R^\infty.\]
Notice, however, that while increasing the dimension $M+1$ of the sliding
window embedding yields better approximations \[SW_{M,\tau} S_N f
(T) \approx SW_{M,\tau} f (T),\] the object being
approximated, $SW_{M,\tau} f(T)$, is changing.
Since $(\R^\infty,\|\cdot\|_2)$ is not complete there is no reason
to believe  this  process converges or stabilizes, even
with a sensible way of comparing, say, $SW_{M,\tau} f(t)$
and $SW_{2M,\frac{\tau}{2}}f(t)$. This is the case, since they are samplings at different rates from the same window.
Perhaps considering the Gromov-Hausdorff distance instead of the
Hausdorff distance would yield such comparison, but at least at the
moment we do not have a natural embedding to make this case.
In addition, even when the metric completion
$\bar {\R^\infty} = \ell^2(\R)$, the space of square-summable
sequences, is well understood, it is also big enough so that
tracking global geometric features requires some work. It is in
situations like this that a succinct and informative summary, such
as persistence diagrams, is critical.

It is known that the space of persistence diagrams is not complete
with respect to the Bottleneck distance, but that it can be completed
by allowing diagrams with countably many points with at most countable
multiplicity, satisfying a natural finiteness condition. See  \cite[Theorem~3.4]{blumberg2012persistent} and \cite[Theorem~6]{mileyko2011probability}. Moreover,
features such as maximum persistence can be easily tracked, and
there is no ambiguity on how to compare the  diagrams
from, say, $SW_{M,\tau} f(T)$ and $SW_{2M,\frac{\tau}{2}} f(T)$.

\begin{proposition} Let $f$ be $L$-periodic,  $N< N'$, $M = 2N$, $M' = 2N'$ and \[ \tau = \frac{2\pi}{L(M+1)} \;\;,\;\; \tau' = \frac{2\pi}{L(M' +1)}.\] If $T\subset \T$ is finite, $Y = SW_{M,\tau} S_N f(T)$ and $Y' = SW_{M',\tau'}S_{N'}f(T)$, then
\[d_{B}\left(\frac{\dgm(Y)}{\sqrt{M+1}},\frac{\dgm(Y')}{\sqrt{M'+1}}\right)\leq  2\Big\|S_Nf - S_{N'}f\Big\|_2\] where $\lambda\cdot \dgm(Z)$ is defined as $\left\{(\lambda x, \lambda y)\;|\; (x,y) \in \dgm(Z)\right\}$, for $\lambda \geq 0$.
\end{proposition}
\begin{proof} Let us fix the notation $\u_n = \u_n(M,\tau)$, $\v_n = \v_n(M,\tau)$, $\u'_n = \u_n(M',\tau')$ and $\v'_n = \v_n(M',\tau')$, in order to specify the dependencies of $\u_n$ and $\v_k$ on
$M$ and $\tau$.  Then we have linear maps
\[
\begin{array}{rccl}
P:&\R^{M' + 1} &\longrightarrow &\R^{M'+1} \\
&\sum\limits_{n=0}^{N'} x_n \u'_n + y_n \v'_n& \mapsto& \sum\limits_{n=0}^{N} x_n \u'_n + y_n \v'_n
\end{array}
\]
\\
\[
\begin{array}{rccl}
Q:&Img(P) & \longrightarrow & \R^{M+1} \\
&\u'_n &\mapsto & \sqrt{\frac{M'+1}{M+1}}\u_n \\ \\
&\v'_n & \mapsto & \sqrt{\frac{M'+1}{M+1}}\v_n
\end{array}
\]
which are well defined by Proposition \ref{propLinInd}.
Moreover, Proposition \ref{propLperd2WindowSize} implies that
$P$ can be interpreted as an orthogonal projection
when restricted to $Y'$, and that $Q$ is an isometry on $P(Y')$.
Notice that for every
$y' \in Y'$  \[\|y' - P(y')\|= \sqrt{\frac{M'+1}{2}\sum_{n=N+1}^{N'} r_n^2 }\]
where  $r_n$ is as defined in Remark \ref{rmkLperiodic}, and therefore
\[d_H(Y',P(Y')) \leq \sqrt{\frac{M'+1}{2}\sum_{n=N+1}^{N'} r_n^2}.\] Finally, since
$Q\c P(Y') = \sqrt{\frac{M'+1}{M+1}}Y$ and $\dgm(\;\cdot\;)$ is
invariant under isometries then
\begin{eqnarray*}
\sqrt{M'+1}\cdot d_B\left(\frac{\dgm(Y')}{\sqrt{M'+1}}, \frac{\dgm(Y)}{\sqrt{M+1}}\right) &=& d_B\big(\dgm(Y'), \dgm(Q\c P(Y'))\big) \\
&=& d_B \big(\dgm(Y'),\dgm(P(Y'))\big) \\ \\
&\leq& 2 d_H(Y',P(Y')) \\
&\leq& \sqrt{2(M'+1)\sum_{n=N+1}^{N'} r_n^2}
\end{eqnarray*} and the result follows  from the identity $r_n = 2\left|\tech{f}(n)\right| = \left|\tech{f}(n)\right| + \left|\tech{f}(-n)\right|$, $n\geq 1$.
\end{proof}

This result, paired with the fact that $\|f - S_N f\|_2 \rightarrow 0$ as $N\to \infty$, and the Structure Theorem \ref{thmStructure}(1,2), imply the following:

\begin{corollary}\label{coroCauchyTruncated}
Let $f\in L^2(\T)$ be $L$-periodic,
$N\in \N$, $\tau_N = \frac{2\pi}{L(2N +1)} $, $T\subset \T$ finite, and let $\bar{Y}_N$ be the set resulting from pointwise centering and normalizing  the point cloud
\[ SW_{2N,\tau_N} S_N f(T) \subset \R^{2N +1 }.\]
Then for any field of coefficients, the sequence $\dgm(\bar{Y}_N)$ of persistence diagrams is Cauchy with respect to $d_B$.
\end{corollary}

Completeness of the set of (generalized) diagrams \cite[Theorem~3.4]{blumberg2012persistent}, allows one to make the following definition:

\begin{definition} Let $w = \frac{2\pi}{L}$ and denote by  $\dgm_\infty(f,T,w)$ the limit in the Bottleneck distance of the sequence $\dgm(\bar{Y}_N)$.
\end{definition}

We hope the notation $\dgm_\infty(f,T,w)$ is suggestive enough
 to evoke the idea that there exists a limiting  diagram from the
sequence of pointwise-centered and normalized versions of
 $SW_{M,\tau}f(T)$, as $M\to \infty$, and while keeping the window size
 $M\tau = \frac{M}{M+1}w\approx w$.
The  First Convergence Theorem (\ref{thmConvergenceI}) bellow, asserts
the validity of this notation.
Before presenting the proof,  we start with a technical result:

\begin{proposition}\label{propUniLimit} Let $f\in C(\T)$ be $L$-periodic, $N\in \N$ and  $\tau_N = \frac{2\pi}{L(2N+1)}$. Then
\begin{equation}\label{eqUnifLimit}
\lim_{N\to \infty}\frac{\left\|C\big(SW_{2N,\tau_N} f(t)\big)\right\|}{\sqrt{2N +1}} =
\left\|f - \tech{f}(0)\right\|_{2}
\end{equation} uniformly in $t\in \T$.
\end{proposition}
\begin{proof} Since the result is trivially true if $f$ is constant, let us assume
$f \neq \tech{f}(0)$ and let
\[g(t) = \frac{f(t) - \tech{f}(0)}{\left\|f - \tech{f}(0)\right\|_2}\]
It follows that $g\in C(\T)$ is $L$-periodic,
\[\tech{g}(0)  = \frac{1}{2\pi}\int_0^{2\pi} g(t)dt =0\;\; \mbox{ and } \;\; \|g\|_2 = \frac{1}{\sqrt{2\pi}} \left(\int_0^{2\pi} |g(t)|^2 dt \right)^{1/2}=1.\]
Using the identity $L(2N+1)\tau_N= 2\pi$, a Riemann sums argument, and the
fact that $g$ is $L$-periodic, it follows that if
\[c_N(t) = \frac{g(t) + g(t + \tau_N) + \cdots + g(t + 2N\tau_N)}{ 2N+1}\] then for all $t\in \T$
\begin{eqnarray*}
\lim_{N\to \infty} c_N(t)
&=& \lim_{\tau_N \to 0} \frac{L}{2\pi} \Big(\tau_N g(t) + \tau_Ng(t + \tau_N) + \cdots + \tau_N g(t + 2N\tau_N)\Big) \\
&=& \frac{L}{2\pi} \int_{t}^{t + \frac{2\pi}{L}} g(r) dr \\
&=& \frac{1}{2\pi} \int_0^{2\pi} g(r)dr \\
&=& 0.
\end{eqnarray*}
We contend that the convergence $c_N (t)\to 0$ is uniform in $t\in \T$.
Indeed, the fact that $g$ is uniformly continuous implies that the sequence $c_N(t)$ is uniformly equicontinuous.
This means that for every $\epsilon >0$ there exists $\delta>0$ independent of $N$ such that for every $t,t'\in\T$ and  all $N\in \N$  \[\mbox{$|t-t'|<\delta$ implies }|c_N(t) - c_N(t')|<\frac{\epsilon}{2}.\]
Let $N_t\in \N$, for $t\in \T$, be such that
$N\geq N_t$ implies $|c_N(t)| <\frac{\epsilon}{2}$.
These two inequalities together  imply that
if $N\geq N_t$ and  $|t-t'| < \delta$ then $|c_N(t')| < \epsilon$.

By choosing a finite open cover of $[0,2\pi]$ with intervals of length $\delta$ and letting $N_0$  be the maximum of the $N_t$'s corresponding to their centers, we get that $N\geq N_0$ implies $|c_N(t)| <\epsilon$ for all $t\in \T$.
Thus the convergence $c_N(t)\rightarrow 0$ is uniform.
A similar argument shows that
\begin{eqnarray*}
\lim_{N\to \infty}  \frac{\left\|C\big(SW_{2N,\tau_N} g(t)\big)\right\|^2}{2N+1}
&=& \lim_{\tau_N\to 0} \frac{L}{2\pi}\sum_{n=0}^{2N} \tau_N\big(g(t + n\tau_N) - c_N(t)\big)^2 \\ \\
&=& \frac{1}{2\pi} \int_0^{2\pi} g(r)^2 dr =1
\end{eqnarray*}
uniformly in $t\in \T$, and replacing $g$ by $\frac{f -\tech{f}(0)}{\left\|f - \tech{f}(0)\right\|_2}$ yields the result.
\end{proof}

\begin{remark} Notice that an alternative proof of  Proposition \ref{propUniLimit}
follows from combining the Structure Theorem \ref{thmStructure}(2),
Parseval's Theorem, and
 the fact that
\[\| S_N f\|_2^2 - \tech{f}(0)^2 = \left\|S_N\left(f - \tech{f}(0)\right)\right\|^2_2\]
\end{remark}

\begin{theorem}[\textbf{Convergence I}]\label{thmConvergenceI}
Let $f\in C^1(\T)$ be an $L$-periodic function,  $N\in \N$,
$\tau_N = \frac{2\pi}{L(2N +1)}$,
$T\subset \T$ finite, and let $\bar{Y}_N$ be as in Corollary \ref{coroCauchyTruncated}.
Let $\bar{X}_N$ be the set resulting from pointwise centering and normalizing the point cloud
\[SW_{2N,\tau_N} f(T) \subset \R^{2N+1}.\]
Then for any  field of coefficients, the sequence $\dgm(\bar{X}_N)$ of persistence diagrams is Cauchy with respect to $d_B$, and
\[\lim_{N\to \infty} \dgm(\bar{X}_N) = \lim_{N\to \infty} \dgm(\bar{Y}_N) = \dgm_\infty(f,T,w).\]
\end{theorem}

\begin{proof}
To prove Theorem \ref{thmConvergenceI} we will
use the Approximation Theorem \ref{thmApproximation} to  show that
\[\lim\limits_{N\to \infty} d_B\Big(\dgm(\bar{X}_N),\dgm(\bar{Y}_N)\Big) =0\]
and combine this with Corollary \ref{coroCauchyTruncated} to obtain the result.

Assume without loss of generality that $f$ satisfies $\tech{f}(0)=0$
and $\|f\|_2 =1$.
Let $X_N $ and $Y_N $ be the resulting sets from pointwise centering the point clouds $SW_{2N,\tau_N} f(T)$ and $SW_{2N,\tau_N} S_Nf(T)$, respectively.
Using the uniform convergence in equation \ref{eqUnifLimit} we get that
\[\lim\limits_{N\to \infty} d_H \left(\bar{X}_N, \frac{X_N}{\sqrt{2N +1}}\right)= 0\] and moreover, since $\lim\limits_{N\to \infty} \|S_N f\|_2 = \|f\|_2 =1$ then
\[\lim_{N\to \infty} d_H \left(\frac{X_N}{\sqrt{2N +1}},\frac{X_N}{\sqrt{2N+1}\|S_Nf\|_2}\right)=0.\]
Now, from the Structure Theorem \ref{thmStructure}(2) we have the identity \[\bar{Y}_N = \frac{Y_N}{\sqrt{2N+1}\|S_Nf\|_2}\] and using the Approximation Theorem \ref{thmApproximation}(1), along with the fact that $C$ is distance non-increasing, we conclude that
\[\lim_{N\to \infty} d_H\left(\frac{X_N}{\sqrt{2N+1}\|S_N f\|_2},\bar{Y}_N\right) \leq \lim_{N\to \infty} \frac{\sqrt{2}\cdot\|R_N f^\prime\|_2}{\|S_N f\|_2\cdot\sqrt{N+1}} =0.\]
The triangular inequality then implies that $\lim\limits_{N\to \infty} d_H(\bar{X}_N,\bar{Y}_N)=0$ and the result follows from combining the stability of $d_B$ with respect to $d_H$, and Corollary \ref{coroCauchyTruncated}.
\end{proof}


The first convergence theorem asserts that for each choice of discretization
$T\subset \T$ one obtains a limiting diagram $\dgm_\infty(f,T,w)$,
by letting $N\to \infty$ in the pointwise centered and normalized versions of
either $SW_{2N,\tau_N} S_N f(T)$ or $SW_{2N,\tau_N} f(T)$.
Next we will show that there is also convergence when $T$ tends to $\T$,
with respect to the Hausdorff distance on subspaces of $\T$.

\begin{theorem}[\textbf{Convergence II}]\label{thmConvergenceII} Let $T,T^\prime \subset \T$ be finite,
and let $f\in C^1(\T)$  be $L$-periodic with modulus of continuity
$\omega:[0,\infty]\longrightarrow [0,\infty]$.
If $w = \frac{2\pi}{L}$
then
\[
d_B \big(\dgm_\infty(f,T,w), \dgm_\infty(f,T^\prime,w) \big)
\leq
2\left\|f - \tech{f}(0)\right\|_2\omega\left(d_H\big(T,T^\prime\big)\right)
\]
and thus there exists a persistence diagram $\dgm_\infty(f,w)$ so that
\[\lim\limits_{T\to \T} \dgm_\infty(f,T,w) = \dgm_\infty(f,w).\]
\end{theorem}
\begin{proof}
Fix $t\in T $ and $t^\prime \in T^\prime$. If we let $\x_N = SW_{2N,\tau_N} f(t)$,
$\x_N^\prime = SW_{2N,\tau_N} f(t^\prime)$, $\tau_N = \frac{2\pi}{L(2N+1)}$ and $\lambda=\left\|f - \tech{f}(0) \right\|_2$
then
\begin{eqnarray*}
\left\|\frac{C(\x_N)}{\|C(\x_N)\|} - \frac{C(\x_N^\prime)}{\|C(\x_N^\prime)\|} \right\|
&\leq&
\left\|\frac{C(\x_N)}{\|C(\x_N)\|} - \frac{\lambda C(\x_N)}{\sqrt{2N+1}}\right\| \;
+ \;
\frac{\lambda\left\|C(\x_N) - C(\x_N^\prime)\right\|}{\sqrt{2N+1}}\;
\\ \\&& +
\left\|\frac{C(\x_N^\prime)}{\|C(\x_N^\prime)\|} - \frac{\lambda C(\x_N^\prime)}{\sqrt{2N+1}}\right\|
\end{eqnarray*}
It follows from Proposition \ref{propUniLimit} that both the summand
\[
\left\|\frac{C(\x_N)}{\|C(\x_N)\|} - \frac{\lambda C(\x_N)}{\sqrt{2N+1}}\right\|
= \frac{\|C(\x_N)\|}{\sqrt{2N+1}}\cdot \left|\frac{\sqrt{2N+1}}{\|C(\x_N)\|} - \lambda \right|
\]
and its version with $\x^\prime_N$, go to zero as $N\to \infty$.
Thus given $\epsilon > 0$ there exists
$N_0 \in \N$ so that $N\geq N_0$ implies
\begin{eqnarray*}
\left\|\frac{C(\x_N)}{\|C(\x_N)\|} - \frac{C(\x_N^\prime)}{\|C(\x_N^\prime)\|} \right\|
&\leq&  \frac{\epsilon}{2} +  \frac{\lambda \left\|C(\x_N) - C(\x^\prime_N)\right\|}{\sqrt{2N+1}} \\ \\
&\leq& \frac{\epsilon}{2} + \frac{\lambda\|\x_N - \x_N^\prime\|}{\sqrt{2N+1}} \\ \\
&=& \frac{\epsilon}{2} + \lambda \left(\sum_{n=0}^{2N}\frac{\big|f(t + n\tau_N) - f(t^\prime + n\tau_N)\big|^2}{2N+1}\right)^{1/2} \\ \\
&\leq& \frac{\epsilon}{2} + \lambda \omega(|t - t^\prime|).
\end{eqnarray*}
Let $\bar{X}_N$ and $\bar{X^\prime}_N$ be the sets resulting from pointwise centering
and normalizing $SW_{2N,\tau_N} f(T)$ and $SW_{2N,\tau_N} f(T^\prime)$, respectively.
Since the estimates above are uniform in $t$ and $t^\prime$ (by Proposition \ref{propUniLimit}), it follows that whenever $N\geq N_0$ then
\[d_H\big(\bar{X}_N , \bar{X^\prime}_N\big) \leq \frac{\epsilon}{2} + \lambda \omega\big(d_H(T,T^\prime)\big).\]
Notice that the Hausdorff distance on the left hand side is for subspaces of
$\R^{2N+1}$, while the one on the right is between subspaces of $\T$.

Applying the Stability Theorem for persistence diagrams yields
\[d_B\big(\dgm(\bar{X}_N), \dgm(\bar{X^\prime}_N)\big)
\leq \epsilon + 2\lambda \omega\big(d_H(T,T^\prime)\big)\]
which by letting $N\to \infty$ and applying the First Convergence Theorem (\ref{thmConvergenceI}), implies
\[d_B\big(\dgm_\infty(f,T,w),\dgm_\infty(f,T^\prime,w)\big) \leq \epsilon
+ 2\lambda\omega\big(d_H(T,T^\prime)\big).\]
Since this is true for any $\epsilon >0$, letting $\epsilon \downarrow 0$ yields
the first part of the theorem.
The existence of $\dgm_\infty(f,w)$ follows from the fact that the set of
generalized persistence diagrams is complete with respect to $d_B$.
\end{proof}

\subsection{A lower bound for maximum persistence}

The Structure Theorem \ref{thmStructure}(3) and the fact that orthogonal projections are
distance non-increasing, allow us to now prove the following:

\begin{theorem}\label{thmWindSize2MaxPers} Let $f\in C^1(\T)$ be an $L$-periodic function, $N\in\N$, $M\geq 2N$,
$L(M+1)\tau = 2\pi$ and let $T\subset \T$ be finite. Furthermore, assume that $d_H(T,\T) <\delta$
for some (see Theorem \ref{thmStructure}(3))
\[
0<\delta <  \max_{1\leq n\leq N} \frac{\sqrt{3}\gor{r}_n}{\kappa_N}\hspace{.2cm}, \hspace{.5cm}
\mbox{where }\;\;\kappa_N = \frac{2\sqrt{2}\left\|S_N f^\prime\right\|_2}{\left\|S_N \left(f - \tech{f}(0)\right)\right\|_2}
\]
Let $\bar{Y} = \bar{Y}_N$ be the set resulting from pointwise centering and normalizing the point
cloud \[SW_{M,\tau}S_N f(T) \subset \R^{M+1},\]
and let $p>N$ be a prime.
If $\dgm(\bar{Y})$ denotes the 1-dimensional $\F_p$-persistence diagram for the Rips
 filtration on $\bar{Y}$, then $\varphi_\tau$  yields an element
 $\x_\varphi \in \dgm(\bar{Y})$  with
\begin{enumerate}
\item $\birth(\x_\varphi) \leq \delta \kappa_N$
\vskip .1in
\item $\death(\x_\varphi) \geq \sqrt{3} \max\limits_{1\leq n \leq N} \gor{r}_n$
\vskip .1in
\end{enumerate}
and therefore \begin{equation}\label{eqMaxPersBound}
mp\Big(\dgm(\bar{Y})\Big)\geq \left(\sqrt{3} \max_{1\leq n \leq N} \gor{r}_n\right) - \delta \kappa_N
\end{equation}

\end{theorem}
\begin{proof}  Given the linear decomposition
\[\varphi_\tau(t) \;\;\;= \sum_{\scriptsize\shortstack{$n=1$\\ $n\equiv 0$ (mod $L$)}}^N\gor{r}_n\big(\cos(nt)\gor{\x}_n + \sin(nt)\gor{\y}_n\big)\] of $\varphi_\tau(t)$ with respect to the orthonormal set $\big\{\gor{\x}_n,\gor{\y}_n\;|\; 1\leq n\leq N, \; n\equiv 0 \;(\mbox{mod  }L)\big\}$
described in the proof of Theorem \ref{thmStructure}(3), it follows that
\[
\begin{array}{rccl}
P_n : &\bar{Y} &\longrightarrow &\C \\
&\varphi_\tau(t)&\mapsto& \gor{r}_n e^{int}
\end{array}
\]
can be regarded as the restriction to $\bar{Y}$ of the orthogonal projection from $\R^{M+1}$ onto $Span\{\gor{\x}_n,\gor{\y}_n\}$.
Since orthogonal projections are linear and norm-non-increasing, then  $\|P_n(\x) - P_n(\y)\| \leq \|\x - \y\|$ for every $\x,\y \in \bar{Y}$.
Thus, if \[ S^1(\gor{r}_n) = \{ \gor{r}_n e^{int}\;|\; t\in T\}\]  it follows that $P_n$ induces  simplicial maps
\[
\begin{array}{rccl}
P_{n_\sharp}: &R_\epsilon \left(\bar{Y}\right)& \longrightarrow &R_\epsilon\left(S^1(\gor{r}_n)\right)  \\
&[\x_0,\ldots,\x_k]&\mapsto& [P_n(\x_0),\ldots, P_n(\x_k)]
\end{array}
\]
for every $\epsilon >0$, which in turn yield homomorphisms
\[P_{n_*} : H_k\left(R_\epsilon \left(\bar{Y}\right);\F_p\right) \longrightarrow H_k\left(R_\epsilon\left(S^1(\gor{r}_n)\right);\F_p\right)\]  of $\F_p$-vector spaces at the homology level.
What we contend is that, via the homomorphisms $P_{n_*}$, the maximum 1-dimensional persistence of $\bar{Y}$ can be bounded below by that of $S^1(\gor{r}_n)$.
Indeed, let $\epsilon_1,\epsilon_2>0$ be so that \[\delta \kappa_N < \epsilon_1 < \epsilon_2 < \sqrt{3}\gor{r}_m\] where $m = \arg\max\;\{\gor{r}_n\;|\; 1\leq n \leq N\}$.
If we write \[T = \{t_0 < t_2 < \cdots <  t_J\}\] it follows from $d_H(T,\T)< \delta$ that $|t_{j} - t_{j-1} | < 2\delta$ for all $j= 1, \ldots, J$, and therefore
\begin{eqnarray*}
\|\varphi_\tau(t_j) - \varphi_\tau(t_{j-1})\|^2 &=& \sum_{\scriptsize\shortstack{$n=1$ \\ $n\equiv 0$ (mod $L$)}}^N 2\gor{r}_n^2\Big(1 - \cos\big(n(t_{j} - t_{j-1})\big)\Big) \\
&\leq&\sum_{\scriptsize\shortstack{$n=1$ \\ $n\equiv 0$ (mod $L$)}}^N \gor{r}_n^2\big(n(t_j - t_{j-1})\big)^2 \\
&=& (t_j - t_{j-1})^2\sum_{n=1}^N \frac{4n^2\left|\tech{f}(n)\right|^2}
{\|S_N f\|^2_2 - \tech{f}(0)^2} \\  \\
&=& \frac{(t_j - t_{j-1})^2}{\left\|S_N \left(f - \tech{f}(0)\right )\right\|^2_2} \sum_{1\leq |n| \leq N} 2\left|\tech{f^\prime}(n)\right|^2 \\ \\
&\leq& 8\delta^2\frac{\left\|S_N f^\prime\right\|^2_2}
{\left\|S_N \left(f - \tech{f}(0)\right )\right\|^2_2}\\ \\
&=& (\delta \kappa_N)^2
\end{eqnarray*} The first inequality  is a consequence of the Taylor expansion for $\cos(x)$ around zero, and $f^\prime$ denotes the first derivative of $f$.
Therefore \[\nu = [\varphi_\tau(t_0), \varphi_\tau(t_1)] + \cdots + [\varphi_\tau(t_{J-1}),\varphi_\tau(t_J) ] + [\varphi_\tau(t_J),\varphi_\tau(t_0)]\]
is a 1-dimensional cycle on $R_{\epsilon_1}(\bar{Y})$, and we obtain the homology class \[P_{m_*}([\nu]) \in H_1\left(R_{\epsilon_1}\left( S^1(\gor{r}_m)\right);\F_p\right).\]

Let
$ \{\theta_0 < \theta_1 <\cdots < \theta_{J_m}\} = \left\{t \mbox{ mod }\frac{2\pi}{m}\;|\; t\in T\right\}$
and let $c_j = \gor{r}_m e^{im \theta_j}$.
It follows from a similar calculation that
\[
\|c_j - c_{j-1}\|^2
\;\;\leq \;\;
(\theta_j - \theta_{j-1})^2 \frac{4\left|\tech{f^\prime}(m)\right|^2}
{\left\|S_N \left(f - \tech{f}(0)\right)\right\|^2_2}
\;\;\leq \;\;
(\delta\kappa_N)^2
\]
and therefore the 1-cycle
\[\mu = [c_0,c_1 ] + \cdots + [c_{J_m-1},c_{J_m}] + [c_{J_m}, c_0]\]
is so that its homology class
$[\mu] \in H_1\left(R_{\epsilon_1}\left(S^1(\gor{r}_m)\right);\F_p\right)$
satisfies $i_*([\mu]) \neq 0$, where $i_*$ is the homomorphism induced by the inclusion
\[i: R_{\epsilon_1}\left(S^1(\gor{r}_m)\right) \hookrightarrow R_{\epsilon_2}\left(S^1(\gor{r}_m)\right).\]
Since $P_{m_*}([\nu]) = m[\mu]$,
and given that $1\leq m\leq N < p$ implies that $m$ is invertible in $\F_p$, then
$i_*\circ P_{m_*}([\nu]) \neq 0$.
From the commutativity of the diagram
\[
\begin{diagram}
\node{H_1(R_{\epsilon_1}(\bar{Y});\F_p)} \arrow{e,t}{i_*} \arrow{s,l}{P_{m_*}}
\node{H_1(R_{\epsilon_2}(\bar{Y});\F_p)} \arrow{s,r}{P_{m_*}}\\
\node{H_1\left(R_{\epsilon_1}\left(S^1(\gor{r}_m)\right);\F_p\right)}\arrow{e,b}{i_*}
\node{H_1\left(R_{\epsilon_2}\left(S^1(\gor{r}_m)\right);\F_p\right)}
\end{diagram}
\]
we conclude that $i_*([\nu]) \neq 0$, and thus $[\nu]$ yields an element $\x_\varphi \in \dgm(\bar{Y})$ so that
\[\birth(\x_\varphi) \leq  \epsilon_1\;\;\; \mbox{ and }\;\;\; \death(\x_\varphi) \geq \epsilon_2.\]
Given that this is true for every $\epsilon_1> \delta \kappa_N$ and every $\epsilon_2 < \sqrt{3}\gor{r}_m$, letting $\epsilon_1 \downarrow \delta \kappa_N $ and $\epsilon_2 \uparrow \sqrt{3}\gor{r}_m$ concludes the proof.
\end{proof}

\begin{remark} It is worth noting that in the proof of
Theorem \ref{thmWindSize2MaxPers} one can
replace $\F_p$, $p> N$, by the field of rational numbers $\Q$.
That is, the estimated bound for maximum persistence is valid for all $N\in \N$
and homology with $\Q$ coefficients.
\end{remark}

Equation \ref{eqMaxPersBound}, together with the Convergence Theorems  I (\ref{thmConvergenceI}) and II (\ref{thmConvergenceII}), imply:

\begin{corollary} Let $f\in C^1(\T)$ be an $L$-periodic function
so that $\tech{f}(0)= 0$ and $\|f\|_2 = 1$.
Let $T\subset \T$ be finite and so that $d_H(T,\T) < \delta$ for some

\[0 < \delta < \frac{\sqrt{3}}{ \sqrt{2}\| f^\prime \|_2}\max_{n\in \N}  \left|\tech{f}(n)\right| \]
Then with $\Q$ coefficients, the 1-dimensional persistence diagram $\dgm_\infty(f,T,w)$
satisfies  \[\frac{1}{2}mp\Big(\dgm_\infty(f,T,w)\Big) \geq \sqrt{3}\max\limits_{n\in \N} \left|\tech{f}(n)\right| - \sqrt{2}\delta\|f^\prime\|_2\]
and therefore
\[mp\Big(\dgm_\infty(f,w)\Big) \geq 2\sqrt{3}\max\limits_{n\in \N} \left|\tech{f}(n)\right|.\]
\end{corollary}

\subsection{The field of coefficients}\label{secFieldCoeff} One question worth asking is
whether the lower bound for maximum persistence presented in Theorem \ref{thmWindSize2MaxPers}, is in fact dependent on the field of coefficients.
More generally, one would like to determine if  the full persistence diagram has such dependency.
To this end, let us consider the functions
\begin{eqnarray*}
g_1(t) &=& 0.6\cos(t) + 0.8\cos(2t)\\
g_2(t) &=& 0.8\cos(t) + 0.6\cos(2t).
\end{eqnarray*}

We construct their associated sliding window point clouds, $SW_{M,\tau} g_1(T)$ and $SW_{M,\tau}g_2(T)$, using $M=4$,  $\tau = 2\pi/5$ and $T = \left\{\frac{2\pi k}{150}\;|\; k=0,1,\ldots, 150\right\}$.
After pointwise centering and normalizing, we compute their 1-dimensional persistent
homology with coefficients in $\F_2$ and $\F_3$.
For this, we use a fast implementation of 1-dimensional persistent homology, based
on the Union-Find algorithm and the work of Mischaikow and  Nanda \cite{nanda2013Morse}. Details of this implementation will appear in \cite{DPH12}.
We summarize the results  in figure \ref{figFunctVsDgm}.

\begin{figure}[htb]
\centering
\includegraphics[scale = .53]{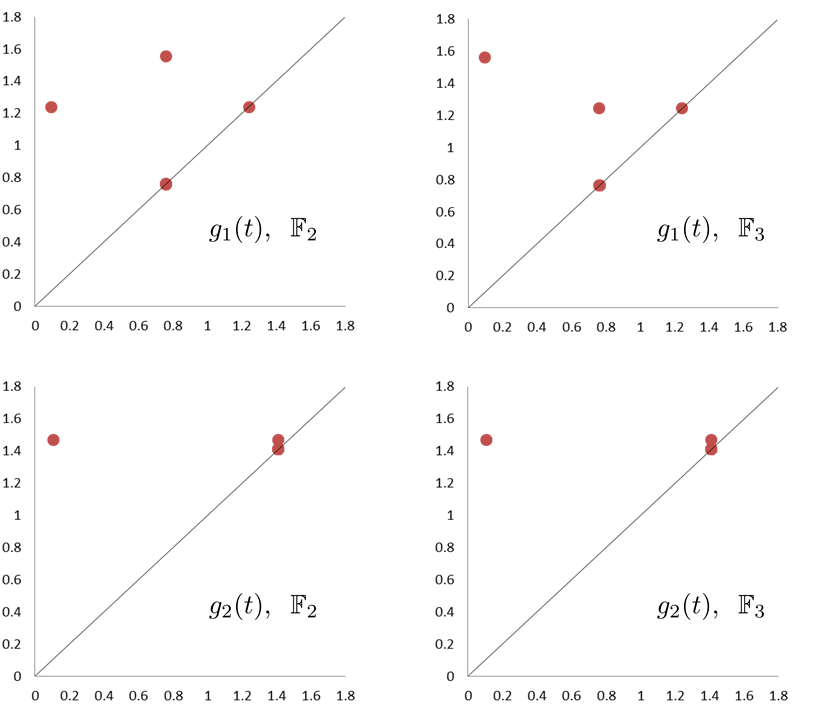}
\caption{1-dimensional $\F_p$-persistence diagrams for the centered and normalized sliding window point clouds on $g_i$. Here the columns correspond to $p =2,3$ and the rows to $i=1,2$.}
\label{figFunctVsDgm}
\end{figure}

\noindent This example shows that, at least in low dimensions, the persistent
homology of sliding window point clouds is coefficient-dependent. Let
us see why this is the case. If $(r_1,r_2) \in \R^2$ is so that
$r_1^2 + r_2^2 = 1$ and $r_1r_2 \neq 0$, it follows from the Structure Theorem \ref{thmStructure}(3) that
if $\alpha_1,\alpha_2 \in [0,2\pi]$ and
\[g(t) = r_1\cos(t - \alpha_1) + r_2 \cos(2t - \alpha_2)\] then for every
$t\in [0,2\pi]$ and $M\geq 4$ ,$\tau = \frac{2\pi}{M+1}$, one has that \[\varphi_\tau(t)= \frac{C\left(SW_{M,\tau} g(t)\right)}{\|C\left(SW_{M,\tau} g(t)\right)\|}\] can be isometrically identified with
\[\gor{\varphi}(t)= \left(r_1 e^{it}, r_2 e^{2it}\right)\in \C^2.\] Let us
use $\gor{\varphi}$ instead of $\varphi_\tau$ for the persistent
homology computation. The first
thing to notice is that the image of $\gor{\varphi}$ can be realized
as the boundary of a M\"{o}bius strip. Indeed, consider the map
\[
\begin{array}{rccl}
\mathcal{M}: & [0,\pi]\times [-1,1]& \longrightarrow & \C^2 \\
&(t,s) &\mapsto & \left(-sr_1e^{it},r_2e^{2it}\right).
\end{array}
\] It follows that $\mathcal{M}$ is a continuous injection on $[0,\pi)\times[-1,1)$, since $r_1r_2 \neq 0$, and that it descends to an embedding of the
quotient space \[\gor{\mathcal{M}} : \big([0,\pi]\times [-1,1]\big/\sim\big)\;\longrightarrow \;\C^2\] where $(0,s) \sim (\pi,-s)$ for every $s\in [-1,1]$. Notice that $[0,\pi]\times [-1,1]/\sim$
serves as the usual model for the M\"{o}bius strip, and
that $\partial \Big(Img(\gor{\mathcal{M}})\Big) = Img(\gor{\varphi})$.

Let
$T = \{t_0 < t_2 < \cdots <  t_J\}$
be a $\delta$-dense  subset of $[0,2\pi]$, $X = \gor{\varphi}(T)$, and let
\[[\nu ] \in H_1\big(R_r (X);\F_2\big)\]
for $r > 4\delta $, be the homology class of the 1-cycle
\[\nu = [\gor{\varphi}(t_0),\gor{\varphi}(t_1)] + \cdots + [\gor{\varphi}(t_{J-1}),\gor{\varphi}(t_J)] + [\gor{\varphi}(t_J),\gor{\varphi}(t_0)].\]
It can be readily checked that if we let
\[V = \left\{(t,s) \; \Big| \; (t,s)\in\big(T \cap [0,\pi)\big)\times\{-1\}\;\; \mbox{ or } \;\; (t+\pi,s)\in\big(T \cap (\pi,2\pi]\big)\times\{1\}\right\},\]
then there
exists  a triangulation of $Img(\gor{\mathcal{M}})$
having $\gor{\mathcal{M}}(V)$ as  vertex set, and so that if we
take coefficients in $\F_2$, then the formal  sum
of its triangles yields a 2-chain  $\Sigma$ with
$\partial_2 (\Sigma ) = \nu $. Moreover, since $T$ is $\delta$-dense
in $[0,2\pi]$ and for all $t\in [0,\pi]$
\begin{eqnarray*}
\left\|\gor{\mathcal{M}}(t,-1) - \gor{\mathcal{M}}(t\pm \delta,1)\right\|^2 &=& 2\big[r_1^2(1 + \cos(\delta)) + r_2^2(1-\cos(2\delta))\big] \\ \\
&\leq& 2\left[r_1^2 \left(2 - \frac{\delta^2}{2}\right) + 2r_2^2\delta^2\right] \\ \\
&=& r_1^2(4 - 5\delta^2) + 4\delta^2
\end{eqnarray*} if $\delta > 0$ is small, then we  can choose $\Sigma$
so that
\[\Sigma \in C_2\big(R_{r'}(X);\F_2\big)\;\;\;, \;\;\; r' = r_1\sqrt{4 - 5\delta^2} + 2\delta. \]
In summary, if
\begin{equation}
\label{eqBoundMobius}
r_1 \sqrt{4 - 5\delta^2} + 2\delta < \sqrt{3}\,r_2
\end{equation}
then the death-time of the class $[\nu]$  is less than or equal to
$r_1 \sqrt{4 - 5\delta^2} + 2\delta$
with coefficients in $\F_2$, but
larger than $\sqrt{3}\,r_2$ (by Theorem \ref{thmWindSize2MaxPers}) with
coefficients in $\F_p$ for any prime $p \geq 3$. Moreover,
with coefficients in $\F_2$ and provided
equation \ref{eqBoundMobius} holds (e.g. for $g_1$),
the first edge across the M\"{o}bius band
prompts the birth of a new class corresponding to the equator
\[t \mapsto \gor{\mathcal{M}}(t,0) = (0,r_2e^{2it})\]
of the embedded M\"{o}bius strip.
This class, in turn,  survives up to $\sqrt{3}r_2$.
With coefficients in $\F_3$, on the other hand, the equatorial and boundary classes will be
in the same persistence class once all the 2-simplices in the M\"{o}bius band
have been added. This results in the death of the class which was born later,
i.e. the one represented by the equator.


\section{Examples: Quantifying Periodicity of Sampled Signals}
\label{secExamples}

%

We present in this section two experiments to test our ideas:
First, the  ranking of signals by periodicity alone, in a way which is invariant to the shape of the periodic pattern;
and second, the accurate classification of  a signal as periodic or  non-periodic at different noise levels.
A detailed description of our methods is provided below, but roughly speaking, we associate to each sampled signal $S = [s_1,\ldots, s_J]$  a  real valued function $f_S$ by cubic spline interpolation, construct its  centered and normalized sliding window point cloud $X_S$, and let
\[\frac{mp\big(\dgm(X_S)\big)}{\sqrt{3}}= Score(S)\]
be its periodicity score.
We then compare it to those obtained with the JTK{\_}CYCLE \cite{Hughes2010}, Lomb-Scargle \cite{Glynn2006,Lomb1976wz,Scargle1982} and  Total Persistent Homology  \cite{CohenSteiner2010} algorithms.

\subsection{Shape Independence}
For this experiment we construct ten different shapes: A 2-periodic pure cosine-like curve, a 2-periodic cosine-like function plus three levels of gaussian noise (variances at 25\%, 50\%, and 75\%  of the signal's amplitude), a noisy saw-tooth (noise level at 25\%  of the signal's amplitude), a function of the form
$\cos(\phi(t))$
for
$\phi(t) = e^{at +b}$,
a noisy and damped cosine-like curve with three periods, a spiky signal with three periods, a noisy square wave with two periods, and a 1-periodic function of the form $Re\left(\sum\limits_{n=1}^5 \tech{f}(n)e^{2int} \right)$
for $\tech{f}(n)$ drawn randomly and uniformly from the unit disk in $\C$.
Each function is then evaluated at 50 evenly spaced time points, yielding the sampled signals $[s_1,\ldots, s_{50}]$ which we input into the algorithms.
For constructing the sliding window point clouds we use  $N=10$, coefficients in $\F_{11}$, $L=2,3,4$ and report the best score.
In all the other algorithms we set the parameters to their suggested or default values.
The results are summarized  in figure \ref{figTestCases}.

\begin{figure}[h]
\includegraphics[scale = .53]{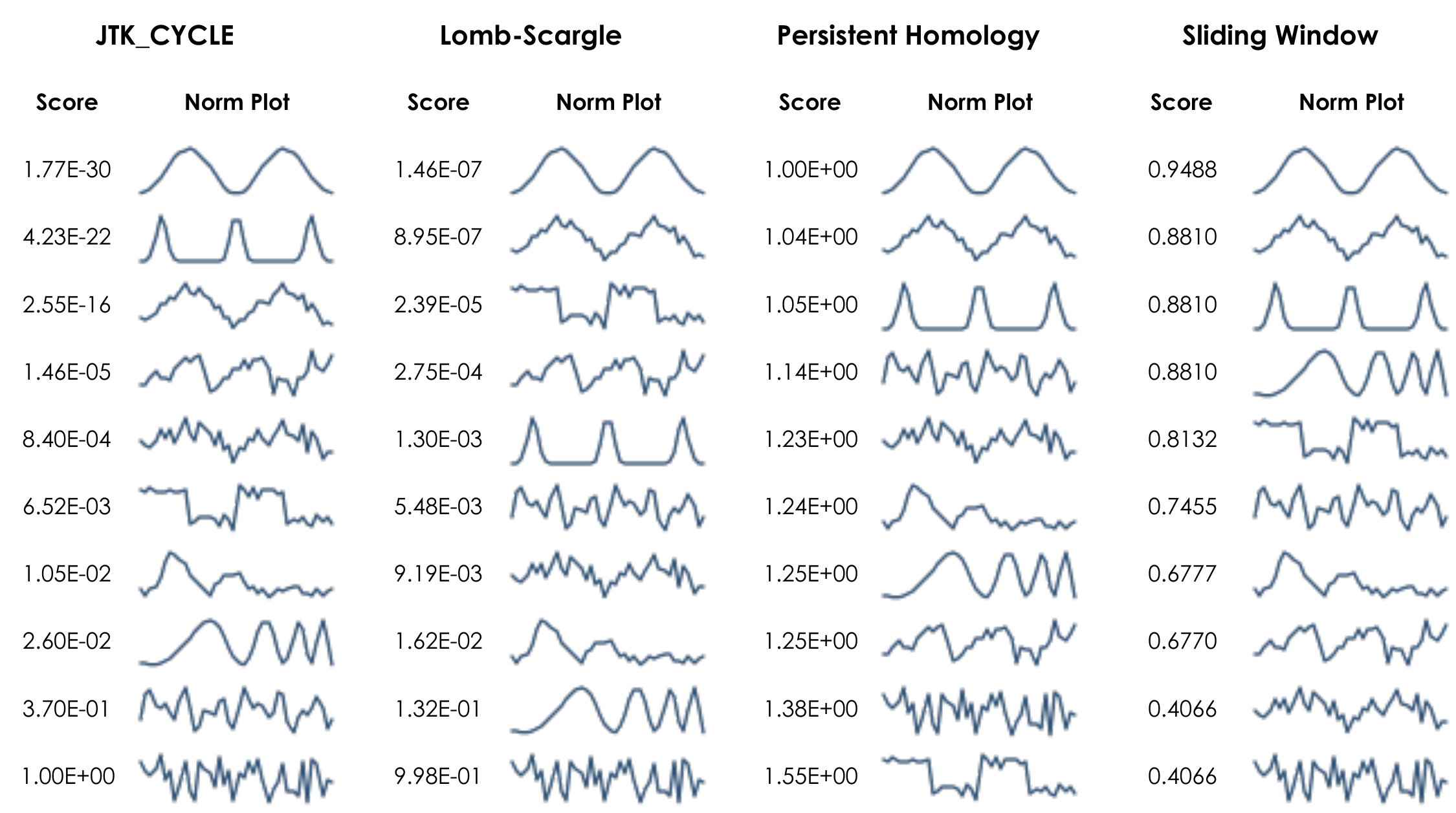}
\caption{Ranking of signals by periodicity. For each algorithm we provide the score and a normalized plot of the signal. The ranking goes from top (highest score) to bottom (lowest score).}
\label{figTestCases}
\end{figure}

Two things are worth noting: First, except for the Sliding Windows method (SW1PerS), all other algorithms have clear preferences for the type of shape they consider to be most periodic.
These biases are of course part of the wiring of the algorithms, and were to be expected.
The second thing to notice has to do with the distribution of scores and their relative differences.
Methods such as JTK or Lomb-Scargle define their periodicity score in terms of $p$-values, which are extremely difficult to interpret.
Our scoring method, by way of contrast, has a clear geometric interpretation and a reasonable distribution.

\subsection{Classification Rates}

We compare the different algorithms by their ability
to separate periodic from non-periodic signals. The performance of
this type of binary classification can be visualized using a Receiver
Operator Characteristic (ROC) plot, which compares the True Positive
Rate (TPR) to the False Positive Rate (FPR) as a cutoff on the scores
is varied.
Here the TPR is the proportion of correctly identified positive
cases out of all positives,
and
FPR is the proportion of negative cases incorrectly identified as
positives out of all the negatives.
The line TPR=FPR is the performance of random guessing; the higher
the ROC curve is above this line, the better its classification performance.
An algorithm that is able to perfectly separate all positive from
negative test cases would have a ROC curve that passes through the
point TPR=1 and FPR=0.
It follows that a reasonable measure of classification success for a particular method, is the area under its ROC curves.

The synthetic data is generated as follows:
The periodic signals (positive cases) span two periods and include a
cosine, cosine with trending, cosine with damping, and cosine with
increased peak steepness. The non-periodic signals (negative cases) include a constant and a linear function.
We generate 100 profiles from each shape by adjusting its phase.
For instance, in the case of the cosine shape we let
\[f_i(t) = \cos\left(2t - \frac{j\pi}{50}\right)\mbox{, }\;\;j=0,\ldots ,99\]
be the profiles.
We sample each of the 600 profiles at 50 evenly spaced time points $t\in [0,2\pi]$,
and add gaussian noise with standard deviation at $0\%$, $25\%$ and $50\%$ of the signal's amplitude.
Please refer to Figure \ref{figSyntheticData} for examples.

\begin{figure}[htb]
\centering
\includegraphics[scale = .42]{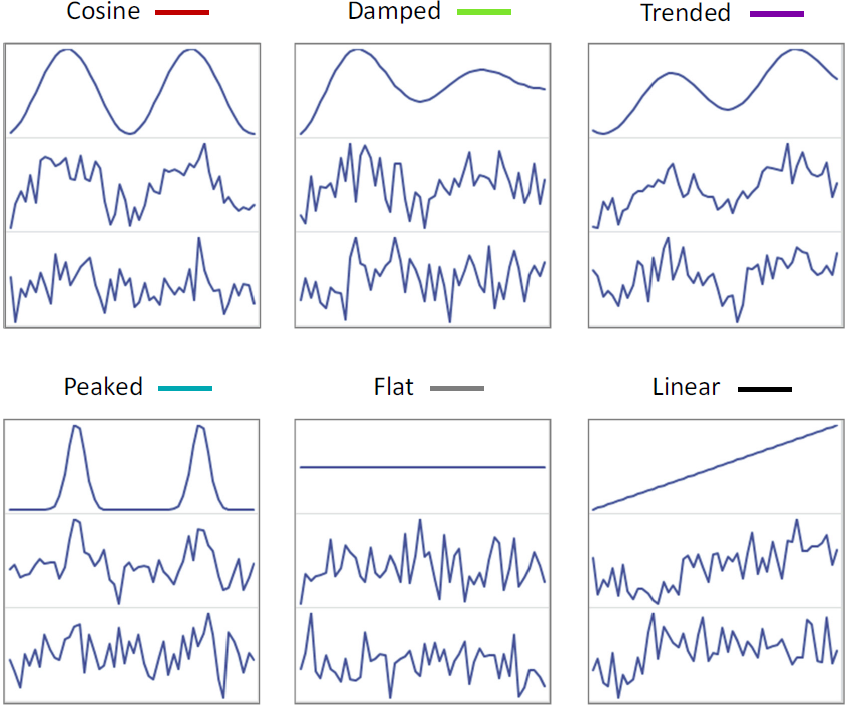}
\caption{Examples of signals in the synthetic data.
We show one signal from each profile type at noise levels 0\%, 25\% and 50\%.}
\label{figSyntheticData}
\end{figure}

\begin{remark}
There are two reason why we regard constant functions as non-periodic.
On the one hand, the intended application for SW1PerS
(Sliding Windows and 1-Persistence Scoring) is to identify
genes that are both relevant and exhibit a periodic expression pattern
 with respect to time.
Relevance in this case means that changes in expression-level translate
into physiological phenomena.
 The second reason has to do
with the philosophy of the proposed method: We quantify periodicity as
the prominence of 1-homology classes in the sliding window point cloud.
Since this point cloud for a constant function is only a point, it does not
have 1-homology and hence is interpreted as coming from a non-periodic
function.
\end{remark}

For the Sliding Windows + 1D $\F_p$-Persistence computation we let $N=10$, $L=2$ and $p=11$.
In order to address noise, we include a layer of (simple) moving average  at the
sampled signal level, and one iteration of mean-shift \cite{comaniciu2002mean} at the sliding window point cloud level.
For the moving average we fix a window size with 7 data points, and use
a cubic spline of this denoised signal to populate the point cloud.
Mean-shift on a pointwise centered and normalized point cloud $X_S$
was implemented as follows: Given a point $x\in X_S$, we let $\bar{x}$ be the
mean of the set
\[
\left\{y \in X_S :  1 - (x\cdot y) < \epsilon\right\}
\]
where $\epsilon = \cos\left(\frac{\pi}{16}\right)$ and $x\cdot y$ denotes the Euclidean inner product of $x$ and $y$.
In other words, $\bar{x}$ is the mean of the $\epsilon$-neighbors of $x$ if
distance is measured with cosine similarity.
We obtain the mean-shifted point cloud

\[\bar{X}_S = \left\{\frac{\bar{x}}{\|\bar{x}\|}\;:\; x\in X_S\right\}\]
which we now use for the persistent homology computation.
We report our results in figure \ref{figRocCurves}.

\begin{figure}[H]
\centering
\includegraphics[scale = .62]{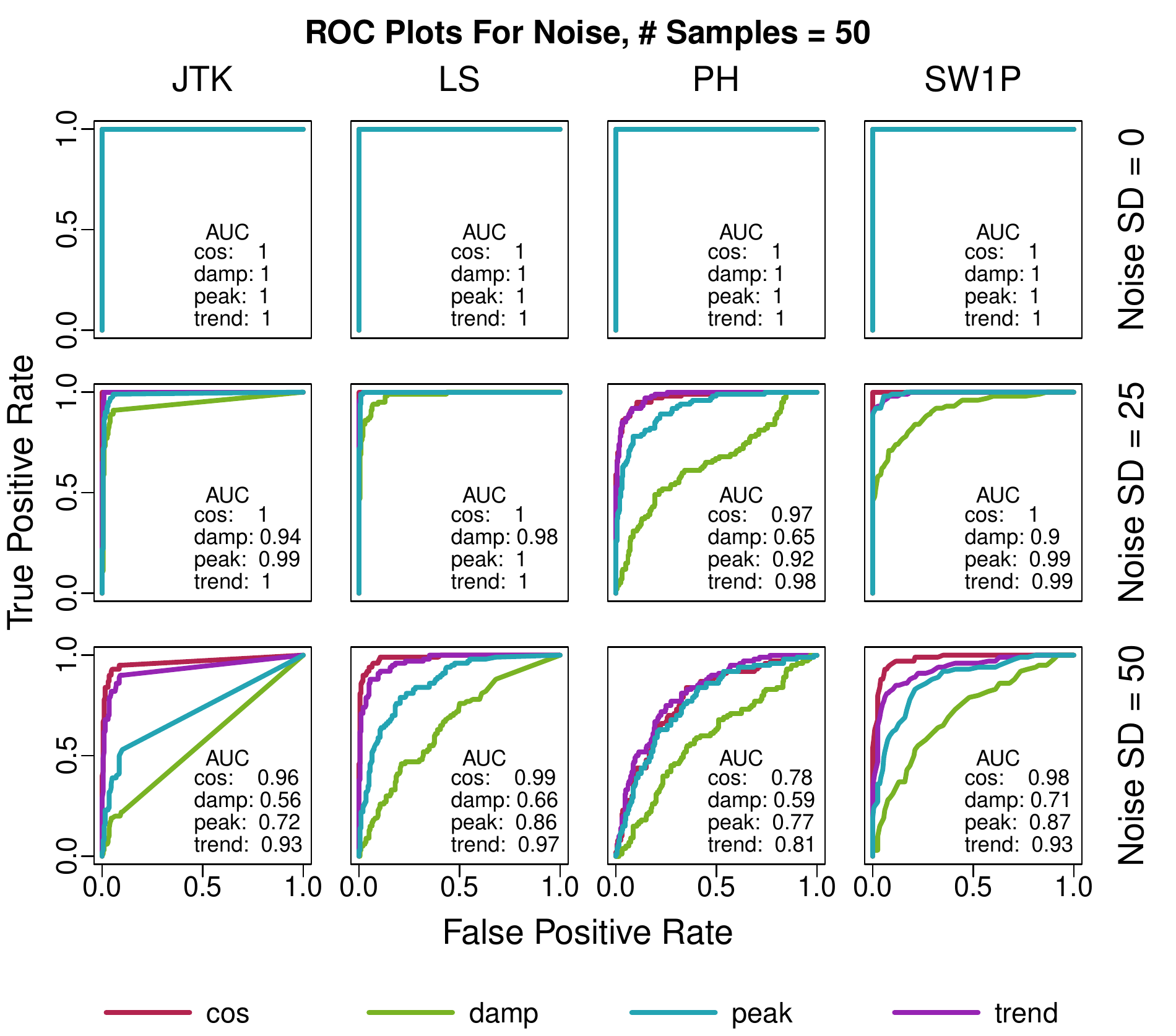}
\caption{ROC curves: True Positive Rate vs False Positive Rate. We compare the classification success of each algorithm on the synthetic data set using the area under its ROC curves. Curves are colored according to the type of periodic shape, and the area under the curve (AUC) is reported. Please refer to an electronic version for colors.}
\label{figRocCurves}
\end{figure}

The Lomb-Scargle periodogram is considered to be one of the best methods for detecting periodicity, and its  ROC curves support this belief.
The fact that it is attuned to favoring cosine-like curves makes it very resilient to dampening, trending and noise.
It was thus a great surprise to see that our method performs comparably well in all cases, except for trended cosines and cosines, and that outperforms it for peaked and damped profiles  at  high noise levels.

A final point we would like to make, is that denoising really is a crucial element
of the SW1PerS pipeline.
We show in Figure \ref{figRocCurvesNoDenoising} the results of quantifying periodicity
on the raw synthetic data, i.e. without applying denoising.
As one can see, in the absence of pre-processing, the results degrade considerably.

\begin{figure}[htb]
\centering
\includegraphics[scale = .6]{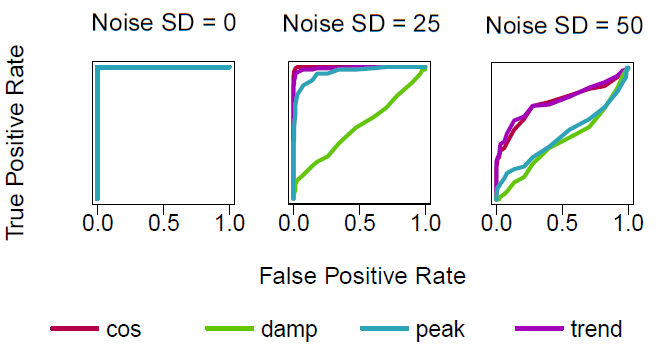}
\caption{ROC curves for the SW1PerS analysis on the raw synthetic data.
That is, we do not apply moving average or mean-shift. For a comparison
on the effect of denoising, please refer to the rightmost column in Figure \ref{figRocCurves}.}
\label{figRocCurvesNoDenoising}
\end{figure}

\section{Final Remarks}

We prove in this paper results which describe the structure of persistence diagrams
obtained by sliding window embeddings of time series.
The main tools for the analysis were a Fourier series approximation argument and
the Stability Theorem for persistence diagrams.
These results we obtained,
provide explicit information about how  diagrams from sliding window point clouds
depend on
the embedding dimension, window size and field of coefficients.
We then present examples of the effectiveness of our method for
 quantifying periodicity of time series data.
The experimental side of this framework will be explored more in depth
 in future work.

This paper also presents the first full {\em theoretical analysis} of the use of
 persistent homology to find structure in time series data.
Time delay embeddings as a means to analyze signals is not new, rather it is a
 well-established method in dynamical systems and in image analysis.
And the use of computation topology methods to find structure in transformed
data has already been considered experimentally before, notably in
\cite{carlsson2008local} and  \cite{de-silva2012Topological}.
This paper, however, is the first to provide a theoretical analysis of the dependency of
persistence on embedding dimension and window size.

There are some interesting new aspects of the use of persistence in this method.
It provides one of the first examples where persistent \textbf{homology} with
coefficients other than $\F_2$ is required.
Other notable examples include \cite{carlsson2008local}, where coefficients
in $\F_3$ were essential to discover the embedded Klein bottle,
and  \cite{de-silva2011Persistent,de-silva2012Topological} where
the authors start with a 1D persistent \textbf{cohomology} class mod
$\F_p$, and lift it to an integer 1D class to get a map to the circle.
They do this by choosing $p$ so that the relevant homomorphism in the Bockstein
long exact sequence is surjective.
Their approach for real data is to choose a prime $p$ at random, and evaluate if
 $H^2(X,\Z)$ has  $p$-torsion.
If it does, they then choose another prime.
By contrast, we have established exactly which primes could be problematic
and can avoid
them in advance.

It was highlighted in Remark \ref{rmkMaxPersDist2Diagonal} that maximum persistence,
the main feature for periodicity we study in this paper, satisfies
$mp(\dgm) = 2d_B(\dgm,\dgm_\Delta)$.
This is in fact part of a bigger picture:
Indeed, the $q$-Wasserstein distance between two persistence diagrams
$\dgm_1$ and $ \dgm_2$
is defined by
\[
W_q(\dgm_1,\dgm_2) =
\min_{\phi} \Big( \sum_{\x \in \dgm_1} ||\x - \phi(\x)||_{\infty}^q \Big)^{\frac{1}{q}}
\]
where $\phi:\dgm_1 \to \dgm_2$ is a matching of $\dgm_1$ with $\dgm_2$.
As with the Bottleneck distance (Remark \ref{rmkMaxPersDist2Diagonal}) one can show that \[W_q(\dgm, \dgm_\Delta) = \frac{1}{2}\Big( \sum_{(x,y)\in\dgm} (y-x)^q\Big)^{\frac{1}{q}}\] and since $W_q \rightarrow d_B$ as $q\to \infty$, then $2W_q(\dgm,\dgm_\Delta)$ can be regarded as a smoother version of $mp(\dgm)$.
When $\dgm$ comes from a sliding window point cloud, $2W_q(\dgm,\dgm_\Delta)$ can be interpreted as a sequence of signatures, or features, for periodicity and other phenomena at the signal level.
Here the parameter $q$ serves as the level of smoothing, and as it gets lager, the emphasis in what $W_q(\dgm,\dgm_\Delta)$ measures shifts from  topological noise and fine attributes to large topological events.

The ring of algebraic functions on persistence diagrams, as a source of features for machine learning purposes, has been recently studied by Adcock et. al.  \cite{Adcock2012Ring}.
We believe these and other signatures such as $W_q(\dgm,\dgm_\Delta)$ should uncover non-trivial signal properties  captured by their sliding window point clouds.
We have devoted this paper to exploring the use of $mp(\dgm)$, but we hope that in future work the list of useful features from persistence diagrams on sliding window point clouds can be extended.

We also mention that Section \ref{secFieldCoeff}
is the first explicit computation of the persistence diagram
of a parametrized space.
The method of Fourier Approximation presented here is
one of the first in a much needed toolbox for explicit computations of
persistence diagrams.

Our final comment is to point out that the fact that the size of the sliding window
should match the period
searched for was not obvious in advance.
Knowing this  provides powerful
 information on sampling density to scientists planning an experiment
 that looks for periodic data, and lays the ground work for the use of SW1PerS
 as a {\em filter} for time series data.

In future work we plan to establish our conjecture that $mp(\dgm)$ is maximized
by our choice of window size, and the main ingredient will be strengthening
the lower bound presented in Theorem \ref{thmWindSize2MaxPers}.
We also plan to establish the filtering properties of SW1PerS
and apply it to a variety of data, including biological data like that from
gene expression and physiology, astronomical data, and weather.
Finally, we plan to extend these methods by using other tools from Topological
Data Analysis to find structure and features in time series.


\begin{thebibliography}{50}
\bibitem{Adcock2012Ring} A. Adcock, E. Carlsson and G. Carlsson, \emph{The Ring of Algebraic Functions on Persistence Bar Codes}, Preprint   available at \url{http://comptop.stanford.edu/u/preprints/multitwo.pdf}, 2012.

\bibitem{blumberg2012persistent} A. J. Blumberg, I. Gal, M. A. Mandell and M. Pancia, \emph{Persistent homology for metric measure spaces, and robust statistics for hypothesis testing and confidence intervals}, arXiv preprint \href{http://arxiv.org/pdf/1206.4581v1.pdf}{arXiv:1206.4581}, 2012.


\bibitem{carlsson2009topology} G. Carlsson, \emph{Topology and Data}, Bulletin of the American Mathematical Society, vol 46(2), p.p. 255--308, 2009.

\bibitem{carlsson2008local} G. Carlsson, T. Ishkhanov, V. de Silva and A. Zomorodian, \emph{On the local behavior of spaces of natural images}, International Journal of Computer Vision, vol 7(1), p.p. 1-12, 2008.

\bibitem{CCSGGS} F.\ Chazal, D.\ Cohen-Steiner, M.\ Glisse, L.\ J.\ Guibas, and S.\ Y.\ Oudot, \emph{Proximity of persistence modules and their diagrams}, In \emph{SCG}, p.p. 237-246, 2009.

\bibitem{CEH} D. Cohen-Steiner, H. Edelsbrunner and J. Harer, \emph{Stability of persistence diagrams}, Discrete and Computational Geometry, vol 37(1), p.p. 103-120, 2007.

\bibitem{CohenSteiner2010} D. Cohen-Steiner, H. Edelsbrunner, J. Harer and Y. Mileyko, \emph{Lipschitz Functions Have $L^p$-Stable Persistence}, Foundations of Computational Mathematics, vol 10(2), p.p. 127--139, 2010.

\bibitem{comaniciu2002mean} D. Comaniciu and P. Meer, \emph{Mean shift: A robust approach toward feature space analysis}, Pattern Analysis and Machine Intelligence,  vol 24(5),  p.p. 603--619, 2002.

\bibitem{de-silva2011Persistent} V. de Silva, D. Morozov and M. Vejdemo-Johansson, \emph{Persistent Cohomology and Circular Coordinates}, Discrete \& Computational Geometry, vol 45(4), p.p. 737--759, 2011.

\bibitem{de-silva2012Topological} V. de Silva, P. Skraba and M. Vejdemo-Johansson, \emph{Topological Analysis of Recurrent Systems}, Workshop on Algebraic Topology and Machine Learning, NIPS 2012, Preprint available at \url{http://sites.google.com/site/nips2012topology/contributed-talks}.


\bibitem{DAOHHH12} A.\ Deckard, R.\ Analfi, D.\ Orlando, J.\ Hogenesch, S.\ Haase and J.\ Harer,
\emph{Design and Analysis of Large-Scale Biological Rhythm Studies:
A Comparison of Algorithms for Detecting Periodic Signals in Biological Data}, Bioinformatics, btt541v1-btt541, 2013.


\bibitem{EH10} H.\ Edelsbrunner and J.\ Harer {\bf Computational Topology, an Introduction}, American Mathematical Society, (2010) (241 pages).

\bibitem{Glynn2006} E. F. Glynn, J. Chen and  A. Mushegian, \emph{Detecting periodic patterns in unevenly spaced gene expression time series using Lomb--Scargle periodograms}, Bioinformatics, vol 22(3), p.p. 310--316, 2006.


\bibitem{Hat02} A.\ Hatcher {\bf Algebraic Topology.} Cambridge Univ. Press, England, 2002.

\bibitem{Hughes2010} M. E. Hughes, J. B. Hogenesch and K. Kornacker, \emph{JTK{\_}CYCLE: An Efficient Nonparametric Algorithm for Detecting Rhythmic Components in Genome-Scale Data Sets}, Journal of Biological Rhythms, vol 25(372), p.p. 372--380, 2010.


\bibitem{kantz2003} H. Kantz and T. Schreiber, \textbf{Nonlinear Time Series Analysis}, Cambridge University Press, 2003.

\bibitem{kim1999nonlinear} H. S. Kim, R. Eykholt and J. D. Salas, \emph{Nonlinear dynaimcs, delay times, and embedding windows}, Physica D: Nonlinear Phenomena, vol 127(1), p.p. 48-60, 1999.

\bibitem{Lomb1976wz} N. R. Lomb, \emph{Least-squares frequency analysis of unequally spaced data}, Astrophysics and Space Science, vol 39, p.p. 447--462, 1976.


\bibitem{mileyko2011probability} Y. Mileyko, S. Mukherjee, and J. Harer,  \emph{Probability measures on the space of persistence diagrams}, Inverse Problems, 27(12), p.p. 124007, 2011.

\bibitem{nanda2013Morse} K. Mischaikow and V. Nanda, \emph{Morse Theory for Filtrations and Efficient Computation of Persistent Homology}, To appear on Discrete and Computational Geometry, 2013.

\bibitem{Mun84} J.\ R.\ Munkres
{\bf Elements of Algebraic Topology.}
Addison-Wesley, Redwood City, California, 1984.

\bibitem{pinsky} M. A.\ Pinsky, \textbf{Introduction to Fourier Anlysis and Wavelets}, The Brooks/Cole Series in Advanced Mathematics, USA, 2003.

\bibitem{DPH12} J. A.\ Perea, A.\ Deckard, S.\ B.\ Haase and J.\ Harer, \emph{SW1PerS: Sliding Windows and 1-Persistence Scoring; Discovering Periodicity in Gene Expression Time Series Data}, preprint (2013).

\bibitem{Scargle1982} J. D. Scargle, \emph{Studies in astronomical time series analysis. II-Statistical aspects of spectral analysis of unevenly spaced data}, Astrophysical Journal, vol 263, p.p. 835--853, 1982.

\bibitem{Stam05} C.\ J.\ Stam, \emph{Nonlinear dynamical analysis of EEG and MEG: Review of an emerging field}, Clinical Neurophisiology , 116, p.p. 2266-2301, 2005.

\bibitem{Hundewale12} N.\  Hundewale, \emph{The application of methods of nonlinear dynamics for ECG in Normal Sinus Rythm}, Int. J. of Computer Science, 9, p.p. 458-467, 2012.

\bibitem{Shannon} C.\ E.\ Shannon, \emph{Communication in the presence of noise}, Proc. Inst. Radio Eng., vol. 37, no. 1, pp. 10-21,  1949.

\bibitem{T81} F.\ Takens, \emph{Detecting strange attractors in turbulence}.
in D.\ A.\ Rand and L.\ -S.\ Young. {\bf Dynamical Systems and Turbulence},
Lecture Notes in Mathematics, vol. 898.
Springer-Verlag. pp. 366--381.

\bibitem{javaplex} A. Tausz, M. Vejdemo-Johansson and H. Adams, \emph{JavaPlex: A research software package for persistent (co)homology}, 2011,  Software available at \url{http://code.google.com/p/javaplex}.

\bibitem{zomorodian2005computing} A. Zomorodian and G. Carlsson, \emph{Computing Persistent Homology}, Discrete \& Computational Geometry, vol 33(2), p.p. 249--274, 2005.

\end{thebibliography}
\end{document}